\documentclass{amsart}

\usepackage{amssymb}
\usepackage{amsthm}
\usepackage{mathrsfs}
\usepackage{enumerate}

\usepackage{rotating}

\usepackage[all]{xy}

\usepackage{fancyhdr}
\usepackage{xr-hyper}

\usepackage[colorlinks=true]{hyperref}

\newtheorem{theorem}{Theorem}[section] 
\newtheorem{proposition}[theorem]{Proposition}
\newtheorem{lemma}[theorem]{Lemma}

\newtheorem{corollary}[theorem]{Corollary}
\newtheorem{defnthm}[theorem]{Definition-Theorem}

\theoremstyle{definition}
\newtheorem{definition}[theorem]{Definition}
\newtheorem{notation}[theorem]{Notation}

\newtheorem{construction}[theorem]{Construction}

\newtheorem{construction-notation}[theorem]{Construction-Notation}

\theoremstyle{remark}
\newtheorem{remark}[theorem]{Remark}
\newtheorem{remarks}[theorem]{Remarks}

\numberwithin{equation}{section}

\newcounter{myenum}

\makeatletter
\renewcommand\thetheorem{\@arabic\c@section.\@arabic\c@theorem}
\newcounter{subtheorem}
 \@addtoreset{subtheorem}{theorem}
\renewcommand\thesubtheorem{\thetheorem.\@arabic\c@subtheorem}
\newcommand\subtheorem{\stepcounter{subtheorem}\par\noindent\protect\textbf{\textup{(\thesubtheorem)}}\quad}
\makeatother

\newcommand{\cal}{\mathcal}
\newcommand{\spec}{\text{Spec }}
\newcommand{\mf}{\mathfrak}
\newcommand{\dotimes}{{\otimes}^{\textbf{L}}}
\newcommand{\comment}[1]{}
\newcommand{\noi}{\noindent}
\newcommand{\scr}{\mathscr}
\newcommand{\bb}{\mathbb}
\newcommand{\into}{\hookrightarrow}
\newcommand{\xym}{\xymatrix}
\newcommand{\sss}{\smallskip}
\DeclareMathOperator{\ord}{ord}
\DeclareMathOperator{\Tor}{Tor}
\DeclareMathOperator{\supp}{Supp}
\DeclareMathOperator{\Hom}{Hom}
\DeclareMathOperator{\ch}{ch}
\DeclareMathOperator{\Td}{Td}
\DeclareMathOperator{\Div}{Div}
\DeclareMathOperator{\pic}{Pic}

\DeclareMathOperator{\parf}{Parf}
\DeclareMathOperator{\coh}{Coh}
\DeclareMathOperator{\cdiv}{CDiv}

\begin{document}

\title{The Hilbert-Chow morphism and the incidence divisor}
\author{Joseph Ross}
\address{Universit\"at Duisburg-Essen, Campus Essen, Fachbereich Mathematik}
\address{http://www.esaga.uni-due.de/joseph.ross/}
\email{joseph.ross@uni-due.de}



\date{September 29, 2010}

\maketitle
\begin{abstract} 

For a smooth projective variety $P$, we construct a Cartier divisor supported on the incidence locus in $\mathscr{C}_a (P) \times \mathscr{C}_{\dim(P) -a - 1}(P)$.  There is a natural definition of the corresponding line bundle on a product of Hilbert schemes, and we show this bundle descends to the Chow varieties.  This answers a question posed by Mazur.
\end{abstract}

\section{Introduction}

Let $(P, \cal{O}_P(1))$ be a smooth projective variety of dimension $n$ over an algebraically closed field $k$.  A \textit{$d$-dimensional algebraic cycle} on $P$ is a finite formal linear combination $\sum a_i Z_i$, where the $Z_i$s are $d$-dimensional integral closed subschemes of $P$ and the $a_i$s are integers.  The \textit{degree} of the $d$-dimensional cycle $\sum a_i Z_i$ is the integer $\sum a_i (\deg_{\cal{O}_P(1)}Z_i )$; and $\sum a_i Z_i$ is \textit{effective} if the $a_i$s are nonnegative.  The \textit{Chow variety} $\mathscr{C}_{d,d^\prime}(P)$ parameterizes algebraic cycles on $P$.  In particular, there is a bijection between $\mathscr{C}_{d,d^\prime}(P)(k)$ and the set of effective algebraic cycles on $P$ of dimension $d$ and degree $d^\prime$.  When the degree plays no essential role, we write $\scr{C}_d(P)$ for the disjoint union $\coprod_{d'} \scr{C}_{d,d'}(P)$ with the understanding we may, if necessary, work one discrete invariant at a time.  A general line of inquiry is to understand how the geometry of $P$ is reflected in moduli spaces associated to $P$.  In particular, for nonnegative integers $a$ and $b$ such that $a+b+1 = n$, one would like to understand the structure of the \textit{incidence locus} $\mathscr{I} = \{ (A,B) | A \cap B \neq \emptyset \} \hookrightarrow \mathscr{C}_a(P) \times \mathscr{C}_b(P)$.  

\sss Over $\mathbb{C}$, Mazur constructs a Weil divisor supported on the incidence locus as follows.  Consider the diagram of schemes:

$$\xymatrix{
P \times \mathscr{C}_a(P) \times \mathscr{C}_b(P) \ar[r]^-{\Delta} \ar[d]^-{{pr}_{23}} & P \times P \times \mathscr{C}_a(P) \times \mathscr{C}_b(P) \\
\mathscr{C}_a(P) \times \mathscr{C}_b(P) \\ }$$

\noi Let $U_a, U_b$ denote the universal cycles on $P \times \mathscr{C}_a(P)$, $P \times \mathscr{C}_b(P)$ respectively (these exist in characteristic zero).  Since $\Delta$ is a local complete intersection morphism, there is a refined Gysin homomorphism ${\Delta}^!$, as constructed in \cite[6.2]{Ful}.  Using standard operations in intersection theory, one has ${{pr}_{23}}_\ast {\Delta}^! (U_a \boxtimes U_b)$, a cycle of codimension 1 on $\mathscr{C}_a(P) \times \mathscr{C}_b(P)$.  The main question of \cite{M} is whether ${{pr}_{23}}_\ast {\Delta}^! (U_a \boxtimes U_b)$ is Cartier.  The main result of this paper is a positive answer to Mazur's question.

\medskip

\begin{theorem} \label{nice main thm} Let $U \subset \scr{C}_a(P) \times \scr{C}_b(P)$ denote the locus of disjoint cycles, i.e., the complement of the incidence locus $\scr{I}$.  Let $U' \subset \overline{U}$ denote the union of products $C_a \times C_b$ of irreducible components $C_a \subset \scr{C}_a(P), C_b \subset \scr{C}_b(P)$ over which the universal cycles intersect properly. 
\begin{itemize}
\item There is a Cartier divisor $D$ on $\overline{U}$ which is supported on $\overline{U} - U$.
\item The restriction of $D$ to $U'$ is an effective Cartier divisor.
\item Let $T$ be the spectrum of a discrete valuation ring $R \supset k$, with generic point $\eta$.  Let $g: T \to \overline{U} \subset \scr{C}_a(P) \times \scr{C}_b(P)$ be a morphism corresponding to cycles $Z,W$ on $P \times T$, and such that $g(\eta) \in U$.  Let $s_D$ denote the canonical section of the line bundle $\cal{O}_{\overline{U}}(D)$.  Then we have
$$ \ord g^*(s_D) = \deg (Z \cdot W) \in \bb{Z},$$
\end{itemize}
where $Z \cdot W \in A_0(P \times T)$ is the class constructed in \cite[20.2]{Ful}.
\end{theorem}

\begin{remark} Our methods will suggest there is a line bundle on the whole of the product $\scr{C}_a(P) \times \scr{C}_b(P)$, but it does not seem reasonable to expect a Cartier divisor beyond the locus $\overline{U}$.  On the locus $U'$, the operation $\Delta^!$ is defined on the cycle level, and all of the coefficients appearing in ${{pr}_{23}}_\ast {\Delta}^! (U_a \boxtimes U_b)$ are positive.  On the locus $\overline{U} - U'$, negative coefficients may appear.
\end{remark}

\sss \textbf{Techniques.} Our approach to Mazur's question is to define the incidence line bundle $\cal{L}$ on a product of Hilbert schemes mapping to the corresponding Chow varieties, and then show $\cal{L}$ is the pullback of a line bundle $\cal{M}$ on the Chow varieties.  Our $\cal{L}$ will be equipped with a canonical nonvanishing rational section on the locus of disjoint subschemes, and we will show this section is induced by a trivialization of $\cal{M}$ on $U$.  Briefly, $\cal{L}$ is the determinant of a perfect complex formed from the universal flat families.  Then we form a proper hypercovering of $\overline{U}$ along the Hilbert-Chow morphism, and a descent datum for $\cal{L}$ on this hypercovering.  This amounts to an identification $\phi$ between two pullbacks of $\cal{L}$, satisfying a cocycle condition.

\sss At first we define the descent datum $\phi$ over a normal base provided the incidence has the expected dimension ($\ref{varphi proper HC}$); this boils down to the Serre Tor-formula for intersection multiplicities and basic properties of the determinant functor (additivity on short exact sequences).  To extend the descent datum over families with more complicated incidence structure, we establish some moving lemmas to produce local trivializations ($\ref{moving lemma}, \ref{moving lemma 2}$), then apply Grothendieck-Riemann-Roch to show these local sections glue ($\ref{GRR move}, \ref{proper compat moving}$).  A useful tool is the following result, which characterizes functions on a seminormal scheme ($\ref{char sn rings}$; see also Definition $\ref{defn ptwise}$): a Noetherian ring $A$ is seminormal if and only if every pointwise function on $\spec A$ which varies algebraically along (complete) DVRs is induced by an element of $A$. 

\sss As for the effectiveness of $({\cal{L}}, \phi)$, i.e., that $\cal{L}$ is induced by a line bundle on the Chow varieties, an outgrowth of $\ref{char sn rings}$ is a criterion for effective descent ($\ref{sn pic inj}$) applicable to our Hilbert-Chow hypercovering: the bundle $\cal{L}$ descends to ${\cal{M}} \in \text{Pic}(\overline{U})$ if it can be trivialized locally on $\overline{U}$, compatibly with the descent datum $\phi$.  The compatible local trivializations are built into the definition of the descent datum. 

\sss \textbf{Motivation.} In the classical construction, the Chow variety $\mathscr{C}_{d,d'}(\mathbb{P}^n)$ is realized as a closed subvariety of the scheme of Cartier divisors of the Grassmannian $\mathscr{G}$ of $(n-d-1)$-planes: to a $d$-dimensional cycle $Z$ on $\mathbb{P}^n$ we associate the (codimension one) set of $(n-d-1)$-planes in $\mathbb{P}^n$ which intersect $Z$.  Thus the natural ample line bundle on $\mathscr{G} \times \text{CDiv}(\mathscr{G})$ simultaneously shows the projectivity of the Chow variety, and endows the incidence locus (in $\mathscr{C}_{d,d'}(\mathbb{P}^n) \times \mathscr{G}$, a special case of the $\mathscr{I}$ considered above) with the structure of an effective Cartier divisor.  This generalizes to the case $\mathscr{C}_d(\mathbb{P}^n) \times \mathscr{C}_{n-d-1}(\mathbb{P}^n)$ using the ruled join; see \cite{PROJ}.

\sss This direct geometric construction does not extend to general smooth projective $P$.  However, the Hilbert scheme (the moduli space for closed subschemes of $P$) and the Hilbert-Chow morphism $\scr{H} \to \scr{C}$ suggest another approach.  The pullback of the line bundle associated to the incidence divisor via $\mathscr{H}(\bb{P}^n) \times \mathscr{G} \to \mathscr{C}(\bb{P}^n) \times \mathscr{G}$ is the determinant of a perfect complex formed from the universal flat families (see the end of Section $\ref{secHC}$), and the determinant construction can be defined for any smooth projective $P$.  Thus one is naturally led to wonder, for a general pair of Hilbert schemes parameterizing subschemes of dimension $a,b$ as above, whether the determinant line bundle descends to the corresponding product of Chow varieties.  The direct geometric construction for $P=\bb{P}^n$ and the determinant formula are in fact compatible; see the end of Section $\ref{secHC}$.

\sss Further motivation comes from the case of zero-cycles and divisors, where the Hilbert-Chow morphism admits a reasonably explicit description.  The equality of the families of zero-cycles associated to two families of zero-dimensional subschemes has a natural expression in terms of determinants; and similarly two families of codimension one subschemes determine the same family of cycles if the determinants of their structure sheaves agree.  For a detailed study of the determinant bundle in the case of zero-cycles and divisors, in particular the descent to the Chow varieties, see \cite{ZD}.

\medskip \textbf{Contents.} In Section $\ref{Kprelim}$ we recall background material on determinant functors and $K$-theory.  In Section $\ref{secHC}$ we discuss the relevant properties of the Chow variety and the Hilbert-Chow morphism, and define the incidence line bundle and the Hilbert-Chow hypercovering along which the incidence bundle descends.  In Section $\ref{secSN}$ we explain the role of seminormality both in defining the descent datum and demonstrating its effectiveness.  In Section $\ref{secDD}$ we construct the descent datum and show it is effective.

\medskip \textbf{Other work.} In an attempt to answer Mazur's question, Wang \cite{W} uses the Archimedean height pairing on algebraic cycles to show that $(n-1)! \mathscr{I}$ is Cartier (over $\mathbb{C}$).  (See the references in \cite{W} for history of the height pairing.)  Given disjoint cycles $A,B$ on $P$ as above, one has the pairing $\langle A, B \rangle := \int_A [G_B]$ defined by integrating a normalized Green's current for $B$ over $A$.  Wang views $\langle A, B \rangle$ as a function on the open set $U$ in $\mathscr{C}_a(P) \times \mathscr{C}_b(P)$ consisting of disjoint cycles, and by studying the behavior of the function as the cycles collide, obtains \cite[Thm. 1.1.2]{W} a metrized line bundle $L$ on $\overline{U}$ and a rational section $s$ that is regular and nowhere zero on $U$, such that $\log {|| s(A,B) ||}^2 = (\dim(P) -1 )! \langle A, B \rangle$.  Using different methods, namely relative fundamental classes in Deligne cohomology, Barlet and Kaddar \cite{BK} associate to a family of cycles over a base $S$ a Cartier (incidence) divisor on $S$ (again over $\mathbb{C}$).  It would be interesting to ``go back" from the Chow varieties to the height pairing.

\medskip \textbf{Conventions.}  We use the definition of the Chow variety from \cite{Kol}.  In characteristic 0, there is a functor of effective algebraic cycles (of dimension $d$ and degree $d'$) defined on the category of seminormal $k$-schemes; and this functor is represented by a seminormal, projective $k$-scheme $\mathscr{C}_{d,d^\prime}(P)$ \cite[I.3.21]{Kol}.  In characteristic $p > 0$, there are several plausible notions of a family of effective algebraic cycles, stemming from the ambiguity of the field of definition of a cycle \cite[I.4.11]{Kol}.  In this case we work with the seminormal, projective $k$-scheme $\mathscr{C}_{d,d^\prime}(P)$ constructed in \cite[I.4.13]{Kol}.  This coarsely represents at least two reasonable functors of effective algebraic cycles.  All schemes considered in this paper are locally Noetherian.  A variety over a field $k$ is an integral separated scheme of finite type over $k$. 

\medskip \textbf{Acknowledgments.} This paper is derived from and improves upon the author's PhD thesis. The author thanks his thesis advisor Aise Johan de Jong.  This work was completed while the author was a wissenschaftlicher Mitarbeiter at the Universit\"at Duisburg-Essen.

\section{Determinant functors and K-groups} \label{Kprelim}

In this section, we briefly recall the notion of a determinant functor, which assigns to a perfect complex an invertible sheaf; and also associated Cartier divisors.  The main reference is \cite{KM}; see also \cite{Fog} and \cite{Knu}.  Then we quickly review some background material on $K$-groups. 

\begin{notation} Let $X$ be a scheme.  Let $D(X)$ denote the derived category of the abelian category $\text{Mod}(X)$ of ${\cal{O}}_X$-modules.  We denote by $D^+(X) \subset D(X)$ the full subcategory of bounded below complexes of ${\cal{O}}_X$-modules, and similarly we have $D^-(X)$ and $D^b(X) = D^+(X) \cap D^-(X)$.  We denote by $D_{(q)coh}(X)$ the full triangulated subcategory of $D(X)$ consisting of pseudo-(quasi)coherent complexes, and by $D^*_{(q)coh}(X)$ the corresponding bounded category for $* = +, -, b$.  We denote by $\parf(X) \subset D^b(X)$ the full triangulated subcategory consisting of perfect complexes. Finally, let $\text{Parf-is}(X)$ denote the category whose objects are perfect complexes on $X$, with morphisms isomorphisms in $D(X)$, and let $\pic(X)$ denote the (Picard) category whose objects are invertible sheaves on $X$, and whose morphisms are isomorphisms. \end{notation}

\sss \textbf{Determinants.} \label{det functor}The main result of \cite[Thm.~2]{KM}\comment{page 42} is that there exists (up to canonical isomorphism) a unique \textit{determinant functor} $\det_X : \text{Parf-is}(X) \to \pic(X)$ extending the usual determinant (top exterior power) of a locally free sheaf.  Indeed the idea is to (locally) replace a perfect complex by a bounded complex of locally free sheaves, and take the tensor product (with signs) of the (usual) determinants of the locally free terms; and show this patches to give a global invertible sheaf.  For every true triangle of complexes $0 \to {\cal{F}}_1 \xrightarrow{\alpha} {\cal{F}}_2 \xrightarrow{\beta} {\cal{F}}_3 \to 0$ in $\text{Parf-is}(X)$, we require an isomorphism $i_X(\alpha, \beta) : \det ({\cal{F}}_1) \otimes \det({\cal{F}}_3) \xrightarrow{\sim} \det({\cal{F}}_2)$, and the isomorphisms $i$ (extending the obvious $i$ for short exact sequences of locally free sheaves) are required to be compatible with isomorphisms of triangles, and more generally triangles of triangles.  Associated to morphisms of schemes we have base change isomorphisms (`interchanging' the determinant with pullback), and these are required to be compatible with composition of morphisms of schemes.

\begin{remark} When $X$ is reduced, $i$ extends to the class of distinguished triangles, is functorial over isomorphisms (in $D(X)$) of distinguished triangles, and is compatible with distinguished triangles of distinguished triangles \cite[Prop. 7]{KM}.
\end{remark}

\textbf{Associated Cartier divisors.} If $\cal{F} \in \parf(X)$ is acyclic at every $x \in X$ of depth $0$, then \cite[Ch.II]{KM} constructs a Cartier divisor $\Div(\cal{F})$ on $X$ and a canonical isomorphism $\cal{O}_X(\Div(\cal{F})) \cong \det_X (\cal{F})$ extending the trivialization $\cal{O}_{X,x} \cong \det_x \cal{F}$ at $x \in X$ of depth $0$.  The formation of this divisor and isomorphism is additive on short exact sequences \cite[Thm.3(ii)]{KM}, and is compatible with base change $f: X' \to X$ such that $\textbf{L}f^* (\cal{F})$ is acyclic at every $x' \in X'$ of depth $0$ \cite[Thm.3(v)]{KM}.  Furthermore, in case $X$ is normal and $x \in X$ is a point of depth $1$, the coefficient of $\overline{ \{ x\} }$ in the Weil divisor associated to $\Div(\cal{F})$ is the alternating sum of the lengths (at $x$) of the cohomology sheaves, i.e., $\sum_i {(-1)}^i \ell_x (\cal{H}^i(\cal{F}))$.

\sss We mention two further properties implicit in \cite[Thm.3]{KM}.  If the complex $\cal{F}$ is acyclic, then $\Div(\cal{F}) = 0$ and the canonical isomorphism $\cal{O}_X(\Div(\cal{F})) = \cal{O}_X \cong \det_X (\cal{F})$ is the trivialization of the determinant of an acyclic complex \cite[Lemma 2]{KM}.  Finally, the construction is determined by the quasi-isomorphism class of the perfect complex $\cal{F}$: since all filtration levels and subquotients appearing in the canonical filtration (``good truncation") of $\cal{F}$ are generically acyclic so long as $\cal{F}$ is, the additivity implies $\Div(\cal{F}) = \sum_i {(-1)}^i \Div (\cal{H}^i(\cal{F}))$.

\medskip \textbf{$K$-groups.} Let $X$ be a variety.  Then $K_0(X)$ is the Grothendieck group of $X$, generated by coherent sheaves on $X$ with relations for short exact sequences of sheaves; $K^0(X)$ is the Grothendieck group of vector bundles.  When $X$ is regular, $K_0(X) \cong K^0(X)$.  We have also the Chern character $\text{ch}: K_0(X) \to A_\ast(X)$, where $A_\ast(X)$ is the Chow group of cycles on $X$, graded by dimension.  We note $\cal{F} \in \parf(X)$ determines a class in $K_0(X)$ since for any abelian category $\scr{A}$ (e.g., $\coh(X)$), $K_0(\scr{A}) \cong K_0(D^b(\scr{A}))$; the latter group is generated by objects of the triangulated category $D^b(\scr{A})$ with relations for distinguished triangles.

\sss The group $K_0(X)$ has a topological filtration: the subgroup $F_k(K_0(X))$ is generated by those $\cal{F} \in \coh(X)$ such that $\dim(\supp(\cal{F})) \leq k$.  For a proper morphism of schemes $f: X \to T$ we obtain a homomorphism $f_\ast :K_0(X) \to K_0(T)$ sending (the class of) a coherent sheaf to the alternating sum of (the classes of) its higher direct image sheaves.  This preserves the topological filtration.  If $T$ is a point and $\cal{F} \in \coh(X)$, then $\chi(\cal{F}) = \text{ch}(f_\ast(\cal{F}))$.

\section{The Hilbert-Chow morphism and the incidence divisor} \label{secHC}

In this section we define the Chow variety, the Hilbert-Chow morphism, and construct our proper hypercovering.  Then we define the incidence line bundle on the product of Hilbert schemes.

\sss We recall an application of the characterization of seminormal schemes \cite[Prop 5.1]{RossSN}, where it is shown that properties (1)-(5) below characterize the Chow variety.  For properties (6) and (7) we refer to \cite{Kol}.

\begin{defnthm}[Existence of the Chow variety] \label{Chow exists}
Let $P$ be a smooth projective variety over a field $k$.  The Chow variety ${\mathscr{C}}_{d, d^\prime}$ of $P$ is a $k$-scheme with the following properties:
\begin{enumerate}[(1)]
\item It is projective over $k$.
\item It is seminormal.
\item For every point $w \in {\mathscr{C}}_{d, d^\prime}$ there exist purely inseparable field extensions $\kappa(w) \subset L_i$ and cycles $Z_i$ on $P_{L_i}$ such that:
\begin{enumerate}[(a)]
\item $Z_i$ and $Z_j$ are essentially equivalent \cite[I.3.8]{Kol}: they agree as cycles over the perfection $\kappa(w)^{perf}$ of $\kappa(w)$;
\item the intersection of the fields $L_i$ is $\kappa(w)$, which is the Chow field (field of definition of the Chow form in any projective embedding of $P$\comment{; see $\ref{construct chow}$}) of any of the $Z_i$ \cite[I.3.24.1]{Kol}; and
\item for any cycle $Z$ on $P_M$ defined over a subfield $k \subset M \subset \kappa(w)^{perf}$ which agrees with the $Z_i$ over $\kappa(w)^{perf}$ (equivalently, agrees with one $Z_i$), we have $\kappa(w) \subset M$ (the Chow field is the intersection of all fields of definition of the cycle).
\end{enumerate}
\item Points $w$ of ${\mathscr{C}}_{d, d^\prime}$ are in bijective correspondence with systems $(k \subset \kappa(w), \{ \kappa(w) \subset L_i, Z_i \}_{i \in I} )$ up to an obvious equivalence relation.
\item For any DVR $R \supset k$ and any cycle $Z$ on $P_R$ of relative dimension $d$ and degree $d^\prime$ in the generic fiber, we obtain a morphism $g: \spec R \to {\mathscr{C}}_{d, d^\prime}$ such that the generic fiber $Z_{\eta}$ and the special fiber $Z_s$ agree with the systems of cycles of the previous property at $g(\eta)$ and $g(s)$.
\item For any numerical polynomial $q$ of degree $d$ and with leading coefficient $d^\prime / (d!)$, we obtain a morphism (the Hilbert-Chow morphism)
$$FC : {({\mathscr{H}}^q)}^{sn}_{red} \to {\mathscr{C}}_{d, d^\prime}$$
\noi by taking the fundamental cycle of the components of maximal relative dimension $(=d)$\cite[I.6.3.1]{Kol}.  A finite number of ${({\mathscr{H}}^q)}^{sn}_{red}$'s surject onto $ {\mathscr{C}}_{d, d^\prime}$.
\item Let $\eta \in {\mathscr{C}}_{d, d^\prime}$ be a generic point.  Then either $\text{dim} \overline{ \{ \eta \} } = 0$ or there exists a cycle $Z_\eta$ on $P_\eta$ defined over $\kappa(\eta)$.  In particular, if $k$ is perfect then there exists a $Z_\eta$ for every generic point $\eta$ of ${\mathscr{C}}_{d, d^\prime}$ \cite[I.4.14]{Kol}.
\end{enumerate}
\end{defnthm}

\begin{construction} \label{proper hyper} Let $P$ be a smooth projective variety, and $r \in \bb{Z}_{\geq 0}$.  Let $\scr{H}_r'$ denote the Hilbert scheme of $r$-dimensional subschemes of $P$.  Let $\scr{H}_r$ denote the seminormalization of the (closed) subscheme of $\scr{H}_r'$ consisting of subschemes $Z$ such that $Z$ has pure $r$-dimensional support (this is different from the notion of a pure sheaf: $Z$ may have embedded components of smaller dimension so long as they are set-theoretically contained in the top-dimensional components).  We have the product of the Hilbert-Chow morphisms ($\ref{Chow exists}$) $\pi : Y_0 = \scr{H}_a \times \scr{H}_b \to \scr{C}_a \times \scr{C}_b =: C$.  Because seminormalization is a functor, we may form a proper hypercovering $\pi_\bullet: Y_\bullet \to C$ augmented towards $C$ whose $i$-th term $Y_i$ is the seminormalization of $Y_0 \times_C \ldots \times_C Y_0$ ($i+1$ factors), with the (seminormalizations of the) canonical morphisms.  
\end{construction}

\begin{remarks}
\subtheorem We explain property (7) in more detail.  For any positive-dimensional component of $\scr{C}_{d,d'}$, its generic point corresponds to a cycle all of whose coefficients are $1$, i.e., a subscheme \cite[1.4.14]{Kol}.  Hence we can find a (component of some) ${(\scr{H}^q)}^{sn}_{red}$ admitting a birational morphism to that component.
\subtheorem Over a field of characteristic zero, the seminormality of the Chow variety and \cite[Prop. 4.1]{RossSN} imply $\cal{O}_{\scr{C}} = {\pi_\bullet}_\ast (\cal{O}_{X_\bullet})$ for a proper hypercovering $X_\bullet$ augmented towards $\scr{C}$.  In characteristic $p > 0$, we have the characterization of the residue fields on the Chow variety as the intersection of all fields of definition \cite[I.4.5]{Kol}.  So by \cite[Prop. 4.1 Corrigendum]{RossSN}, we have $\cal{O}_{\scr{C}} = {\pi_\bullet}_\ast (\cal{O}_{X_\bullet})$ for a proper hypercovering such that $X_0 = \scr{H}^{sn}_{red}$.

In more detail and in the language of \cite[4.1]{RossSN}, we explain how to construct (locally) a pointwise function on $\scr{C}$ from a pointwise function on $\scr{H}^{sn}_{red}$ which belongs to ${\pi_\bullet}_\ast (\cal{O}_{X_\bullet})$.  So suppose $z \in \scr{C}(P_k)$ corresponds via a morphism $\spec \kappa(z) \to \scr{C}(P_k)$ to the cycle $Z$ on $P_{\kappa(z)}$.  Consider an algebraically closed field $\overline{K}$ containing $\kappa(z)$ and the cycle associated to the base change $Z_{\overline{K}}$.  Then by \cite[I.4.5]{Kol}, the residue field $\kappa(z)$ is characterized as the intersection in $\overline{K}$ of all fields of definition of $Z$, i.e., the intersection of all $E_i$ such that $k \subset E_i \subset \overline{K}$ and there exists a subscheme $Y_i \subset P_{E_i}$ whose associated cycle agrees with $Z_{\overline{K}}$ upon base change.  Consider fields $E_0, E_1$ satisfying these conditions.  Then we have morphisms $\spec E_i \to \scr{H}$ with the property that the compositions $\spec \overline{K} \to \spec E_i \to \scr{H} \to \scr{C}$ are the same.  Thus we have a commutative diagram:

$$\xym{
\spec \overline{K} \ar[r] \ar@<-.15cm>[d] \ar@<.15cm>[d]  & {(\scr{H} \times_\scr{C} \scr{H})}^{sn}_{red} \ar@<-.15cm>[d]_-{p_0} \ar@<.15cm>[d]^-{p_1} \\
\spec E_0, \spec E_1 \ar[r] & \scr{H} \\ }$$

(The morphism from $\spec \overline{K}$ factors through the seminormalization.)  Considering $a \in {\pi_\bullet}_\ast (\cal{O}_{X_\bullet})$ as a pointwise function, we obtain elements $a_i \in E_i$.  The preceding diagram shows $a_0 = a_1$ in $\overline{K}$, therefore $a_0 \in E_0 \cap E_1$.  By the same argument we find $a_0 \in E_0 \cap E_i$ for all $i$, therefore $a_0 \in \kappa(z)$.  Thus we made an element in the residue field $\kappa(z)$.  It varies algebraically along DVRs by \cite[4.1]{RossSN}.

\subtheorem If $X = Y \cup Z$ is a reducible scheme with irreducible components $Y,Z$, then the field of definition of $Y$ is contained in the field of definition of $X$.  Also, a scheme and its seminormalization have the same residue fields.  Hence to cut out the residue fields on $\scr{C}$, it is enough to consider the subscheme $\scr{H}_r \into \scr{H}_r'$ defined in $\ref{proper hyper}$.  So we have $\cal{O}_{\scr{C}} = {\pi_\bullet}_\ast (\cal{O}_{X_\bullet})$ for a proper hypercovering such that $X_0 = \scr{H}_r$.  
\end{remarks}

We record the (presumably known) fact that the Hilbert-Chow morphism is compatible with products.

\begin{lemma} \label{HC products} If $P,P'$ are smooth projective varieties over a field $k$, then the following diagram commutes.  (We suppose $p$ has leading coefficient $d' / (d!)$ and $\deg(p)=d$; and $q$ has leading coefficient $e' / (e!)$ and $\deg(q)=e$.)
$$\xym{ {(\scr{H}^p(P))}^{sn}_{red} \times {(\scr{H}^{q}(P'))}^{sn}_{red} \ar[r]  \ar[d]^-{FC \times FC} & {(\scr{H}^{pq}(P \times P'))}^{sn}_{red} \ar[d]^-{FC} \\
\scr{C}_{d,d'}(P) \times \scr{C}_{e,e'}(P') \ar[r] & \scr{C}_{d+e,d'e'}(P \times P')  \\} $$
\end{lemma}

\begin{proof} We describe the map in the top row: if $Z \into P \times T, Z' \into P' \times T$ constitute a $T$-point of ${(\scr{H}^p(P))}^{sn}_{red} \times {(\scr{H}^{q}(P'))}^{sn}_{red}$, then the scheme theoretic intersection ${pr}^*_{13}Z \cap {pr}^*_{23}Z'$ in $P \times P' \times T$ is a $T$-point of ${(\scr{H}^{pq}(P \times P'))}^{sn}_{red}$.  A top-dimensional component in the product scheme is induced by a pair of top-dimensional components; and length multiplies, so the coefficients in the product cycle are the products of the coefficients of the factors.\comment{(We omit further details.) }\end{proof}

The main goal of this paper is to construct a Cartier divisor supported on the incidence locus.   Now we define an invertible sheaf (the ``incidence bundle") on a product of Hilbert schemes, and show the incidence bundle is pulled back from a bundle on a product of Chow varieties in the case $P=\mathbb{P}^n$.  

\medskip \noi \begin{construction} \label{defn incidence} Let $P$ be a smooth projective variety over any base scheme $B$ (over which we take all fiber products), and let ${\mathscr{H}}^1, {\mathscr{H}}^2$ denote the Hilbert schemes corresponding to numerical polynomials $q_1, q_2$.  Over each ${\mathscr{H}}^i$ we have a universal flat family (a closed subscheme of $P \times {\mathscr{H}}^i$); denote by ${\mathscr{U}}_i$ its pullback to $P \times {\mathscr{H}}^1 \times  {\mathscr{H}}^2$.  Standard facts about the behavior of perfect complexes under certain operations (stability under tensor product; pullback; and pushfoward via a proper morphism of finite Tor-dimension; and the perfectness of a coherent sheaf on the source of a smooth morphism which is flat over the target) imply $\textbf{R} {{pr}_{23}}_\ast ({\cal{O}}_{{\mathscr{U}}_1} \dotimes {\cal{O}}_{{\mathscr{U}}_2})$ is a perfect complex on ${\mathscr{H}}^1 \times {\mathscr{H}}^2$.  For details on the necessary facts about perfect complexes, see section 2 of \cite{ZD}.  The incidence bundle ${\cal{L}}$ is defined to be its determinant:

$${\cal{L}} := \displaystyle \text{det}_{{\mathscr{H}}^1 \times {\mathscr{H}}^2} \textbf{R} {{pr}_{23}}_\ast ({\cal{O}}_{{\mathscr{U}}_1} \dotimes {\cal{O}}_{{\mathscr{U}}_2}) .$$
\end{construction}

\noi In fact we will be interested in this construction only on the locus $Y_0$ defined earlier in this section.  As motivation for pursuing the determinant formula (mentioned in the Introduction), we make contact with the classical construction of the Chow variety of $P=\mathbb{P}^n$.   As explained in the Introduction, the construction of the Chow variety endows the incidence locus $\scr{I} \hookrightarrow \scr{C}_d(\mathbb{P}^n) \times \scr{G}$ with the structure of a Cartier divisor.  Let ${FC}_{\bb{P}^n} : {\mathscr{H}}(\bb{P}^n) \to {\mathscr{C}}(\bb{P}^n)$ denote the Hilbert-Chow morphism (and its product with $\scr{G}$).  In the special case of Construction $\ref{defn incidence}$ with $P=\mathbb{P}^n$ and ${\mathscr{H}}^2 = {\mathscr{G}}$, it follows from \cite[Thms.~1.2, 1.4]{Eis} that there is a canonical isomorphism ${\cal{L}} \cong {{FC}_{\bb{P}^n}}^\ast {\cal{O}}({\mathscr{I}})$ of invertible sheaves on ${\mathscr{H}} \times {\mathscr{G}}$.  Here canonical means the following: over the locus $U_0$ of disjoint subschemes, the complex $\textbf{R} {{pr}_{23}}_\ast ({\cal{O}}_{{\mathscr{U}}_1} \dotimes {\cal{O}}_{{\mathscr{U}}_2})$ is acyclic, hence there is a canonical trivialization $\cal{L} |_{U_0} \cong \cal{O}_{U_0}$.  This rational section is the pullback via ${FC}_{\bb{P}^n}$ of the canonical trivialization of $\cal{O}(\scr{I})$ on the complement of $\scr{I}$.

\section{Seminormal schemes and descent criteria} \label{secSN}

In this section we explain the role of seminormality both in defining the descent datum and demonstrating its effectiveness.

\begin{definition}[\cite{GT}] A ring $A$ is a \emph{Mori ring} if it is reduced and its integral closure $A^\nu$ (in its total quotient ring $Q$) is finite over it; if $A$ is a Mori ring, $A^{sn}$ denotes its seminormalization, the largest subring $A \subset A^{sn} \subset A^\nu$ such that $\spec A^{sn} \to \spec A$ is bijective and all maps on residue fields are isomorphisms.  The seminormalization is described elementwise in \cite[1.1]{Trav}. We say $A$ is seminormal if $A = A^{sn}$ (so we only define seminormality for Mori rings).  A locally Noetherian scheme $X$ is Mori if and only if it has an affine cover by Noetherian Mori rings \cite[Def. 3.1]{GT}.
\end{definition}

\begin{definition} \label{defn ptwise} Let $A$ be a ring, and let $S = \{f_y \in \kappa(y) | y \in \spec A \}$ be a collection of elements, one in each residue field.  Then we say \emph{$S$ is a pointwise function on $\spec A$}.  We say the pointwise function $S$ \emph{varies algebraically along (complete) DVRs} if it has the following property: for every specialization $\mf{p}_1 \subset \mf{p}_2$ in $A$ and every (complete) discrete valuation ring $R$ covering that specialization via a ring homomorphism $g: A \to R$, there exists a (necessarily) unique $f_R \in R$ such that $\overline{g_{\mf{p_1}}}(f_{\mf{p_1}}) = f_R$ (in $K$) and $\overline{g_{{\mf{p_2}}}}(f_{\mf{p_2}}) = \overline{f_R}$ (in $k_0$).
\end{definition}

The main result of \cite[2.2, 2.6]{RossSN} is the following.

\begin{theorem}
\label{char sn rings}
Let $A$ be a seminormal (in particular, Mori) ring which is Noetherian.  Let $\{f_y \in \kappa(y) | y \in \spec A \}$ be a pointwise function on $\spec A$ which varies algebraically along (complete) DVRs.  Then there exists a unique $f \in A$ whose image in $\kappa(y)$ is $f_y$ for all $y \in \spec A$.
\end{theorem}

This simplifies greatly the problem of defining a descent datum for a line bundle on a seminormal scheme, as seen in the following corollary.

\begin{corollary} \label{pic dvrs} Let $X$ be a seminormal locally Noetherian (in particular, Mori) scheme, and let $L, M \in \pic(X)$.  Then an isomorphism $L \cong M$ is equivalent to an ``identification of fibers varying algebraically along DVRs," that is:
\begin{quote}
for any field or (complete) DVR $R$, any $\spec R \xrightarrow{f} X$, an identification $\beta_f: f^\ast L \cong f^\ast M$ compatible with restriction to the closed and generic points: if $s \xrightarrow{i} \spec R, \eta \xrightarrow{j} \spec R$ denote the inclusions, then $\beta_{fi} = i^\ast \beta_f$ and $\beta_{fj} = j^\ast \beta_f$.
\end{quote}
\end{corollary}
\begin{proof}  
Fix an open cover $X = \cup_i \spec S_i$ with $S_i$ a seminormal (Noetherian and Mori) ring which trivializes both $L$ and $M$, and fix trivializations $\varphi_i : L_i := L \arrowvert_{\spec S_i} \cong {\cal{O}}_{\spec S_i}$, $\psi_i : M_i := M \arrowvert_{\spec S_i} \cong {\cal{O}}_{\spec S_i}$.  Then defining $L \cong M$ is equivalent to identifying $\Gamma(\spec S_i, L_i) \cong \Gamma (\spec S_i, M_i)$ as $S_i$-modules (for all $i$), compatibly with restrictions.  Then considering the diagram:
$$\xymatrix{
\Gamma(L_i) \ar[r]^-\cong \ar[d]^-{\varphi_i} & \Gamma(M_i) \ar[d]^-{\psi_i} \\
S_i \ar[r] & S_i \\ }$$
and its pullbacks to spectra of fields and DVRs, we see that relative to the fixed $\varphi_i, \psi_i$, a family $\beta_f$ as in the statement is equivalent to an invertible pointwise function on each $S_i$ varying algebraically along DVRs.  By Theorem $\ref{char sn rings}$ this is equivalent to a family of elements $f_i \in {S_i}^\times = \text{Isom}_{S_i}(S_i, S_i)$.  The $f_i$ thus obtained agree on overlaps by the uniqueness statement in Theorem $\ref{char sn rings}$.  Then using the above diagram again we see that relative to the fixed trivializations, the family $f_i$ is equivalent to a family of isomorphisms $\Gamma(L_i) \cong \Gamma(M_i)$ compatible with restrictions.
\end{proof}

We will need the following general fact later.

\begin{lemma} \label{dense specializations} Let $X$ be a seminormal $k$-scheme and $Y \into X$ a closed subscheme.  Suppose every $y \in Y$ admits a generization to a point in $X - Y$, i.e., that $Y$ does not contain any generic points of $X$.  Let $S$ be a pointwise function on $X$ which varies algebraically along DVRs covering specializations within $X - Y$, and along those from $X-Y$ to $Y$.  Then $S$ varies algebraically along DVRs.\end{lemma}
\begin{proof} This follows readily from the techniques used in \cite{RossSN}.  We may assume $X$ is affine.  Let $\nu : X^\nu \to X$ denote the normalization.  Then our pointwise function $S$ determines a pointwise function on $X^\nu$ which is constant along the fibers of $\nu$.  The normalization is birational, so identifies generic points of $X^\nu$ with those of $X$.  Hence as a pointwise function on $X^\nu$, $S$ varies algebraically along specializations of the form $\eta \leadsto x$, with $\eta$ generic.  This is enough to conclude $S$ is induced by an element of $\Gamma(X^\nu, \cal{O}_{X^\nu})$ (see \cite[2.4]{RossSN}).  But then because $S$ is constant along the fibers of $\nu$, this element comes from $\Gamma(X, \cal{O}_X)$ and \textit{a fortiori} varies algebraically along all DVRs.  \end{proof}

As for the effectiveness of a descent datum, we recall the following results from \cite[3.2, 3.8]{ZD}.  We denote by $({\cal{L}}, \phi)$ an element of $\pic (X_\bullet)$, i.e., ${\cal{L}} \in \pic(X_0)$ and $\phi : {p_0}^\ast {\cal{L}} \xrightarrow{\sim} {p_1}^\ast {\cal{L}}$ is an isomorphism on $X_1$ satisfying the cocycle condition on $X_2$.

\begin{proposition} \cite[3.2]{ZD}
\label{sn pic inj} Let $X$ be a scheme, and let $\pi_\bullet: X_\bullet \to X$ be a proper hypercovering augmented towards $X$ which satisfies ${\cal{O}}_X = {\pi_\bullet}_\ast ({\cal{O}}_{X_\bullet})$.  Then:
\begin{itemize}
\item ${\pi_\bullet}^\ast : \pic(X) \to \pic (X_\bullet)$ is injective; and
\item the image of ${\pi_\bullet}^\ast$ consists of those $({\cal{L}}, \phi)$ satisfying the following property:

\noi for every $x \in X$, there exists an open $U \subset X$ containing $x$ and a trivialization $T_x : {\cal{L}} \arrowvert_{{\pi_0}^{-1}(U)} \xrightarrow{\sim} {\cal{O}}_{{\pi_0}^{-1}(U)}$ compatible with $\phi$ in the sense that the diagram
$$ \xymatrix{
{p_0}^\ast ({\cal{L}} \arrowvert_{{\pi_0}^{-1}(U)} ) \ar[r]^-{{p_0}^\ast T} \ar[d]^-\phi & {\cal{O}}_{{(p_0)}^{-1} ({\pi_0}^{-1}(U))} \ar[d]^-= \\
{p_1}^\ast ({\cal{L}} \arrowvert_{{\pi_0}^{-1}(U)} ) \ar[r]^-{{p_1}^\ast T} & {\cal{O}}_{{(p_1)}^{-1} ({\pi_0}^{-1}(U))} \\ }$$
commutes.
\end{itemize}
\end{proposition}

\begin{definition} Let $X_\bullet \to X$ be a simplicial scheme augmented towards $X$, and let $(\cal{L}, \phi) \in \pic(X_\bullet)$.  We say $(\cal{L},\phi)$ \textit{satisfies the descent property} if for all $x \in X$, there exists an open $U \subset X$ and a trivialization $t : \cal{L} \arrowvert_{\pi_0^{-1}(U)}  \xrightarrow{\sim} \cal{O}_{{\pi_0^{-1}(U)}}$ which is compatible with $\phi$ in the sense of the preceding proposition. \end{definition}

\begin{remark} The Proposition applies when $X$ is seminormal and $X_\bullet$ satisfies any of the conditions in \cite[4.1 Corrigendum]{RossSN}, for example the proper hypercovering $\pi_\bullet : Y_\bullet \to C$ defined in $\ref{proper hyper}$. \end{remark}

\section{Definition of the descent datum} \label{secDD}

In this section we prove the main result, in the following form.  Having established this result, we will consider the refinements and further properties stated in $\ref{nice main thm}$.

\begin{theorem}[$\ref{existence of DD}, \ref{end game}$] \label{main thm intro} With the notation as in $\ref{proper hyper}$, let $\scr{U}_r \into P \times \mathscr{H}_r$ denote the (pullback of the) universal flat family.  Using $\ref{defn incidence}$ we may form the determinant line bundle $\cal{L}$ on $Y_0$.  Now base change everything to $\overline{U} \subset C$, the closure of the locus of disjoint cycles.  Then the following hold.
\begin{itemize}
\item The sheaf $\cal{L}$ lifts to an invertible sheaf on $Y_\bullet$, i.e., there is an isomorphism $\phi : {p_0}^\ast {\cal{L}} \cong {p_1}^\ast {\cal{L}}$ on $Y_1$ satisfying the cocycle condition on $Y_2$.  
\item The descent datum $\phi$ is effective: there is a unique $\cal{M} \in \pic(\overline{U})$ such that $(\pi^*\cal{M}, can) \cong (\cal{L}, \phi)$.
\end{itemize}
\end{theorem}

\subsection{Notation and preliminary reductions}

\begin{definition} Let $P$ be a smooth projective $k$-variety of dimension $n$.  A \textit{Hilbert datum for $P$ over $T$} consists of the following:
\begin{enumerate}
\item a seminormal $k$-scheme $T$;
\item $Z \into P_T := P \times_k T$ a $T$-flat closed subscheme of relative dimension $a$, such that the support of $Z$ has pure dimension $a$ in every fiber; and
\item $W \into P_T$ a $T$-flat closed subscheme of relative dimension $b$, such that the support of $W$ has pure dimension $b$ in every fiber;
\end{enumerate}
such that $a +b +1 \leq n$; and every point $t \in T$ admits a generization to the locus of disjoint subschemes.  Thus a Hilbert datum $(Z,W)$ is simply a morphism $T \to \scr{H}_a \times \scr{H}_b$ such that the image of every generic point of $T$ lies in a component of $\scr{H}_a \times \scr{H}_b$ with at least one pair of disjoint subschemes.

Typically we will make some construction from $(Z,W)$ and then show the construction only depends on $[Z],[W]$, the cycles underlying $Z$ and $W$.  Therefore we make the following definition.  A \textit{Hilbert-Chow datum for $P$ over $T$} is a pair of Hilbert data $(Z,W), (Z', W')$ for $P$ over $T$ such that $[Z] = [Z']$ and $[W] = [W']$.  Since the supports of $Z,W$ are assumed pure-dimensional, we have also $\supp(Z) = \supp(Z')$ and $\supp(W) = \supp(W')$.  Thus a Hilbert-Chow datum $(Z,Z',W,W')$ for $P$ over $T$ is nothing more than a morphism $T \to {(\scr{H}_a \times \scr{H}_b \times_{\scr{C}_a \times \scr{C}_b} \scr{H}_a \times \scr{H}_b)}^{sn}$ such that (after projecting to either $\scr{H}_a \times \scr{H}_b$ factor) every generic point of $T$ lands in a pair of irreducible components with at least one pair of disjoint subschemes.  

Because we work on the subscheme $\scr{H}_r$ of the Hilbert scheme, disjointness of subschemes on $\scr{H}_a \times \scr{H}_b$ corresponds exactly to disjointness of their associated cycles on $\scr{C}_a \times \scr{C}_b$.  In general, two subschemes could have disjoint associated cycles but lower-dimensional components which coincide.  So in the notation above we have  $Z \cap W = \emptyset$ if and only if $Z' \cap W = \emptyset$.

Note that given a morphism $S \to T$ of seminormal $k$-schemes and a Hilbert-Chow datum for $P$ over $T$, by pullback we obtain a Hilbert-Chow datum for $P$ over $S$.   \end{definition}

\begin{notation} The structure morphism $P_T \to T$ will be called $\pi$.

\sss \noi For $\cal{F}, \cal{G} \in \parf(P_T)$, we set  $f_T(\cal{F},\cal{G}) := \det_T \textbf{R} \pi_\ast (\cal{F} \dotimes \cal{G}) \in \pic(T)$.  

\sss \noi If $\alpha$ is a $b$-dimensional cycle on $P_T$ with $\alpha=\sum a_i W_i$, we put $f_T(\cal{O}_Z, \alpha) := \otimes_i ({f_T(\cal{O}_Z, \cal{O}_{W_i})}^{\otimes a_i})$.  In general we use the notation $[-]$ to denote the cycle associated to a subscheme or coherent sheaf: this means the top-dimensional components and their geometric multiplicities, even if, for example, $b < n-a-1$.  (In fact we used this in the preceding definition.)

\sss \noi If $T$ is affine and equal to $\spec R$, we may write $f_R$ for $f_T$. 
 
  \sss \noi We will use the subscripts ${(-)}_0$ and ${(-)}_\eta$ to denote the base change of some object to closed and generic fibers, respectively.
 
 \sss \noi By the incidence $Z \cap W$, we mean the underlying reduced algebraic subset $\supp(Z) \cap \supp(W)$.  Stated properties of $Z \cap W$ will depend only on the underlying supports $\supp(Z), \supp(W)$.
 \end{notation}

\sss \noi \textbf{Goal.}  For every Hilbert-Chow datum, we construct an isomorphism $\phi_T^{Z,Z',W,W'} : f_T(\cal{O}_Z, \cal{O}_W) \cong f_T(\cal{O}_{Z'}, \cal{O}_{W'})$ varying functorially in $T$, and with the descent property.  The essential case is $b=n-a-1$.

\begin{proposition}[reduction to fields and DVRs] \label{field DVR} To define an isomorphism $\phi_T : f_T(\cal{O}_Z, \cal{O}_W) \cong f_T(\cal{O}_{Z'}, \cal{O}_{W'})$ for each Hilbert-Chow datum, for all smooth projective $P$, so that for each $P$, the collection $\{ \phi_T \}$
\begin{enumerate}
\item is compatible with base change $S \to T$; and
\item satisfies the cocycle condition;
\end{enumerate}
it is sufficient to define an isomorphism $\phi_T$ for each Hilbert-Chow datum with $T$ the spectrum of a field or complete DVR, compatible with base change among fields and complete DVRs, and satisfying the cocycle condition on fields.
\end{proposition}

\begin{proof} This is a consequence of $\ref{pic dvrs}$. \end{proof}

\begin{proposition}[reduction to the diagonal] \label{diagonal}To define an isomorphism $\phi_T : f_T(\cal{O}_Z, \cal{O}_W) \cong f_T(\cal{O}_{Z'},\cal{O}_{W'})$ for each Hilbert-Chow datum, for all smooth projective $P$, so that for each $P$, the collection $\{ \phi_T \}$:
\begin{enumerate}
\item is compatible with base change $S \to T$;
\item satisfies the cocycle condition; and
\item satisfies the descent property;
\end{enumerate}
it is sufficient to define an isomorphism $\phi_T$ for each Hilbert-Chow datum with $W=W'$, for all $P$, so that for each $P$, the collection $\{ \phi_T \}$ has the stated properties.
\end{proposition}

\begin{proof} On $P \times P \times T$, let $\cal{O}_\Delta$ denote the structure sheaf of the diagonal ($\times T$), i.e., the image of the closed immersion $P \times T \xrightarrow{\Delta \times 1_T} P \times P \times T$.  Given $T$-flat closed subschemes $Z, W \into P \times T$, we let $Z \times W \into P \times P \times T$ denote the scheme-theoretic intersection ${pr}^*_{13}Z \cap {pr}^*_{23}W$.  Then there is a canonical isomorphism of line bundles on $T$: $f_T(\cal{O}_Z, \cal{O}_W; P) \cong f_T(\cal{O}_{Z \times W}, \cal{O}_\Delta; P \times P)$.  Then the proposition follows from the fact that the Hilbert-Chow morphism is compatible with products ($\ref{HC products}$). \end{proof}

\begin{remark} We may even assume $W$ is constant, i.e., there is a $k$-subscheme $W_k$ such that $W = W_k \times_k T$; and we may assume $W_k$ is integral (even smooth). \end{remark}

\begin{notation} When $W$ and $W'$ are omitted from the notations, this means $W=W'$.   \end{notation} 

We start with some easy cases of our goal.

\begin{lemma} \label{varphi disjoint} Let $(Z,W)$ be a Hilbert datum over any base $T$ such that $Z \cap W = \emptyset$.  Then there is a canonical trivialization $\varphi^Z_T : f_T(\cal{O}_Z, \cal{O}_W) \cong \cal{O}_T$ which is compatible with base change. \end{lemma}
\begin{proof} The hypothesis implies $\cal{O}_Z \dotimes \cal{O}_W$, hence also $\textbf{R} \pi_\ast (\cal{O}_Z \dotimes \cal{O}_W)$, is acyclic.  The pullback via $S \to T$ is also acyclic, and the trivialization of the determinant of an acyclic complex is compatible with base change. \end{proof}

\begin{remark} The canonical isomorphism $\varphi$ of the lemma has an additivity property in each variable.  For example, if $\cal{F}_1 \to \cal{F}_2 \to \cal{F}_3 \to^{+1}$ is a distinguished triangle in $\parf(P_T)$ such that $\supp(\cal{F}_i) \cap W = \emptyset$ for all $i$, then the isomorphism $f_T(\cal{F}_1, \cal{O}_W) \otimes f_T(\cal{F}_3,\cal{O}_W) \cong f_T(\cal{F}_2, \cal{O}_W)$ induced by the triangle corresponds, via the identifications $\varphi_T^{\cal{F}_i}$, to multiplication $\cal{O}_T \otimes \cal{O}_T \to \cal{O}_T$. \end{remark}

\begin{corollary} \label{phi disjoint} Among Hilbert-Chow data satisfying $Z \cap W = Z' \cap W = \emptyset$, there exists a collection of isomorphisms $\phi_T^{Z,Z'} : f_T(\cal{O}_Z, \cal{O}_W) \cong f_T(\cal{O}_{Z'}, \cal{O}_W)$ which is compatible with base change and satisfies the cocycle condition.  \end{corollary}
\begin{proof} We define $\phi_T^{Z,Z'; W} := {(\varphi^{Z'}_T)}^{-1} \circ \varphi^Z_T : f_T(\cal{O}_Z, \cal{O}_W) \cong f_T(\cal{O}_{Z'}, \cal{O}_W)$ to be the composition of the canonical trivializations.  This is compatible with base change because each $\varphi^Z_T$ is.  We check the cocycle condition:
$$\phi_T^{Z',Z''; W} \circ \phi_T^{Z,Z'; W} = ({(\varphi^{Z''}_T)}^{-1} \circ \varphi^{Z'}_T) \circ ({(\varphi^{Z'}_T)}^{-1} \circ \varphi^Z_T) = {(\varphi^{Z''}_T)}^{-1} \circ \varphi^Z_T = \phi_T^{Z,Z''; W}.$$
\end{proof}

From now on we keep the collection $\{ \phi_T \}$ whose existence is asserted in $\ref{phi disjoint}$.  The idea is to gradually extend it to a collection over Hilbert-Chow data with increasingly complicated incidence structure, until we have covered the whole moduli space.  Note that an isomorphism of line bundles on a reduced (e.g., seminormal) scheme is determined by its restriction to generic points (i.e., points of depth 0).  Since our base $T$ is always reduced, when we have defined an isomorphism $\phi$ for a more general class of Hilbert-Chow data, to check agreement with previously defined isomorphisms it suffices to check agreement on generic points.

\begin{lemma} \label{varphi proper H, D adds} Let $(Z,W)$ be a Hilbert datum over  $T$, and suppose that $(Z \cap W)_\eta  = \emptyset$ for all generic points $\eta \in T$.  (This holds, for example, whenever $Z$ and $W$ intersect properly on $P_T$.)  Then there exists a Cartier divisor $D_{Z,W}$ on $T$ and a canonical isomorphism $\varphi^Z_T : f_T(\cal{O}_Z, \cal{O}_W) \cong \cal{O}_T(D_{Z,W})$ characterized by agreeing with the trivialization $\varphi^{Z_\eta}_\eta$ for every generic point $\eta \in T$.  When $Z \cap W = \emptyset$, $D_{Z,W} = 0$ and $\varphi^Z_T$ is the canonical trivialization.  The formation of the divisor $D_{Z,W}$ and the isomorphism $\varphi^Z_T$ are compatible with base change $S \to T$ preserving the generic disjointness.

Furthermore, the formation of $D_{Z,W}$ is additive in each variable: if $\cal{F}_1 \to \cal{F}_2 \to \cal{F}_3 \to^{+1}$ is a distinguished triangle in $\parf(P_T)$ such that $(\supp(\cal{F}_i) \cap W)_\eta = \emptyset$ for all generic points $\eta \in T$, all $i$; then $D_{\cal{F}_1,W} + D_{\cal{F}_3,W} = D_{\cal{F}_2,W}$; and the triangle induces, upon application of $\varphi_T^{\cal{F}_i}$, the canonical isomorphism $\cal{O}_T(D_{\cal{F}_1,W}) \otimes \cal{O}_T(D_{\cal{F}_3,W}) \cong \cal{O}_T(D_{\cal{F}_2,W})$.  Similarly we have an additivity property in the variable $W$.
\end{lemma}

\begin{proof} To see the generic disjointness is satisfied when $Z$ and $W$ intersect properly, note that $Z$ (resp.~$W$) has codimension $\geq b+1$ (resp.~$\geq a+1$) in $P_T$, hence $Z \cap W$ has codimension $\geq a+b+2$ in $P_T$.  Therefore $\dim(Z \cap W) < \dim (T)$, so the support of $\cal{O}_Z \dotimes \cal{O}_W$ cannot dominate any component of $T$.
The hypothesis on the incidence means the construction of \cite[Ch.II]{KM} applies.  The compatibility with base change is a consequence of \cite[Thm.3(v)]{KM}; and the additivity is inherited from \cite[Thm.3(ii)]{KM}.
\end{proof}

\begin{remark} Our essential task is to show that given a Hilbert-Chow datum $(Z,Z',W)$, we have $D_{Z,W} = D_{Z',W}$. \end{remark}

\subsection{Moving lemmas}

\begin{proposition} \label{moving lemma} Let $(Z,W)$ be a Hilbert datum over $T$ the spectrum of a local ring, with $W = W_k \times_k T$ for a $b$-dimensional $k$-scheme $W_k$, and suppose $(Z \cap W)_\eta = \emptyset$ for every generic point $\eta \in T$.  Then there exist subvarieties $B_1, \ldots , B_n \subset P$ of dimension $b+1$, $M_i \in \pic(B_i)$, and short exact sequences:
$$0 \to M_i \xrightarrow{s^i_0} \cal{O}_{B_i} \to Q^i_0 \to 0$$
$$0 \to M_i \xrightarrow{s^i_\infty} \cal{O}_{B_i} \to Q^i_\infty \to 0$$
such that:
\begin{enumerate}
\item $(Z \cap {\supp(Q^i_*)} )_\eta = \emptyset$ for all $i$, all generic $\eta$; and
\item the $b$-dimensional cycle $[W] + \sum_i ([Q^i_0] - [Q^i_\infty])$ is disjoint from $Z$.
\end{enumerate}
\end{proposition}
\begin{proof} We let $Z_0$ denote the cycle over the closed fiber of $T$.  By Chow's moving lemma \cite[Thm.]{Rob}, we can find a cycle $\alpha$ rationally equivalent to $[W_k]$ and satisfying $\alpha \cap Z_0 = \emptyset$; hence also $\alpha \cap Z = \emptyset$ on $P_T$.  This shows we can achieve the second property; the issue is to show we can move $W$ in such a way that the first property is satisfied.

Suppose we have a closed immersion $P \into \bb{P}^{2n+1}$.  Then every step in moving a cycle involves essentially two choices: a linear space $L \cong \bb{P}^n \into \bb{P}^{2n+1}$, disjoint from $P$, from which projection induces a finite morphism $\pi_L : P \to \bb{P}^n$; and an element $g \in PGL(n+1)$.  The excess intersection $e(Z_0, \pi^*_L {\pi_L}_* [W] - [W])$ is smaller than $e(Z_0, [W])$ for generic $L$; and $\pi_L^* (g \cdot {\pi_L}_* W)$ is disjoint from $Z_0$ for generic $g$.

If $(Z \cap W)_\eta = \emptyset$ for all generic points $\eta \in T$, then for generic choices of $L,g$, we have $(Z \cap \pi_L^* (g \cdot {\pi_L}_* [W]) )_\eta = (Z \cap \pi_L^* ( {\pi_L}_*[W]) )_\eta = \emptyset$.  The $Q_*$s are supported in subsets of the form $\pi_L^* (g \cdot {\pi_L}_* W)$ and $\pi_L^* ( {\pi_L}_* W)$, hence the result.

\end{proof}

We need a slight variation for subvarieties $W$ as in $\ref{moving lemma}$ of dimension strictly smaller than $n-a-1$; eventually we need that such subvarieties do not contribute to $D_{Z,W}$.  

\begin{proposition} \label{moving lemma 2} Let $(Z,W)$ be a Hilbert datum over $T$ a base of dimension $\leq 1$, with $W = W_k \times_k T$ for a $k$-scheme $W_k$, and suppose $(Z \cap W)_\eta = \emptyset$ for every generic point $\eta \in T$.  

Suppose further that $\dim(W_k)= b \leq n-a-2$.  Then there exist subvarieties $B_1, \ldots , B_n \subset P$ of dimension $b+1$, $M_i \in \pic(B_i)$, and short exact sequences:
$$0 \to M_i \xrightarrow{s^i_0} \cal{O}_{B_i} \to Q^i_0 \to 0$$
$$0 \to M_i \xrightarrow{s^i_\infty} \cal{O}_{B_i} \to Q^i_\infty \to 0$$
such that:
\begin{enumerate}
\item $(Z \cap B_i )_\eta = \emptyset$ for all $i$, all generic $\eta$; and
\item the $b$-dimensional cycle $[W] + \sum_i ([Q^i_0] - [Q^i_\infty])$ is disjoint from $Z$.
\end{enumerate}
\end{proposition}

\begin{remark} The first condition in the conclusion implies  $(Z \cap {\supp(Q^i_*)} )_\eta = \emptyset$ for all $i$, all generic $\eta$. \end{remark}

\begin{proof} Again we are intersecting a finite number of open conditions.  Without loss of generality we may assume $W_k$ is an integral subscheme of dimension $n-a-2$.  Let $pr_1(Z) \into P$ denote the ``sweep" of the family $Z$ (with the reduced structure); this is a subscheme of dimension $\leq a+1$.

Now $pr_1(Z)$ and $W_k$ are not expected to meet, and we have an open dense $U \subset T$ such that $pr_1(Z_U) \cap W_k = \emptyset$.  For a generic finite morphism $\pi : P \to \bb{P}^n$ (as in the proof of $\ref{moving lemma}$) we have, possibly after shrinking $U$, that $\pi(pr_1(Z_U)) \cap \pi(W_k) = \emptyset$; and that the pair $(Z, \pi^* \pi_* (W_k) - W_k)$ has smaller excess intersection than does $(Z,W_k)$.  Now we move $\pi_*(W_k)$ along a general smooth (affine) rational curve $C \into PGL(n+1)$.  Let $\mathscr{Y} \into \bb{P}^n \times C$ note the total space of the resulting family, 

Write $\scr{Y} = \sum m_i Y_i$, and let ${Y_i}^{fl} \into \bb{P}^n \times \bb{P}^1$ denote the flat limit of the family $Y_i \into \bb{P}^n \times C$.  Then $\scr{Y}^{fl} :=\sum m_i {Y_i}^{fl}$ is the unique way to complete $\scr{Y}$ to a family of cycles over $\bb{P}^1$.  Let  $pr_1(\mathscr{Y}^{fl}) \into \bb{P}^n$ be the sweep; this is a subscheme of dimension $n-a-1$.  Choose some $t \in T$ such that $Z_t \cap W_k = \emptyset$.  For a general choice of $C$, since $\dim(Z_t) + \dim(\scr{Y}^{fl}) = a + (n-a-1) < n$, we will have $\pi(pr_1(Z_t)) \cap \mathscr{Y}^{fl} = \emptyset$.  Hence the disjointness holds on an open dense of $T$.  This process can be iterated until we have a cycle $\alpha \sim W_k$ such that $pr_1(Z) \cap \alpha= \emptyset$.

The subvarieties $B_i \into P$ lie in subsets of the form $\pi^{-1}(\mathscr{Y}^{fl})$.  (This follows from the proof that flat pullback preserves rational equivalence \cite[1.7]{Ful}.)  Since a general $\mathscr{Y}^{fl}$ used in one step of the moving process is disjoint from a general member of the family $Z$, this holds after pullback by $\pi$ as well.  \end{proof}

\subsection{Grothendieck-Riemann-Roch}

\begin{lemma} \label{chisubsub}  Let $T$ be the spectrum of a field, and suppose ${\cal{F}}, {\cal{G}} \in \parf(P_T)$ satisfy $\dim(\supp({\cal{F}})) + \dim(\supp({\cal{G}})) < n  = \dim(P)$.  Then $\chi({\cal{F}} \dotimes {\cal{G}}) = 0$.
\end{lemma}

\begin{proof} Since $F_a(K_0(P))$ is generated by $[{\cal{O}}_V]$, $V \subset P$ a subvariety of dimension $\leq a$ \cite[Ex. 15.1.5]{Ful}\comment{p. 285}, we may assume ${\cal{F}}, {\cal{G}}$ are structure sheaves of subvarieties of dimensions $a,b$ respectively, with $a+b < n$.

\medskip \noi Now since $P$ is smooth, any coherent sheaf has a finite length resolution by finite rank locally free sheaves, so we may apply \cite[18.3.1 (c)]{Ful} to the closed immersion $i : V \to P$ with $\beta = {\cal{O}}_V$.  This gives, in the Chow group $A_\ast(P)$:

$$i_\ast ( \text{ch}({\cal{O}}_V) \cap \text{Td}(V) ) = \text{ch}(i_\ast {\cal{O}}_V) \cap\text{Td}(P) .$$

\medskip \noi As $\text{ch}({\cal{O}}_V) = 1$ and $ \text{Td}(V) = [V] + r_V$ with $r_V \in A_{< a}(V)$, the left hand side lies in $A_{\leq a}(P)$.  

\medskip \noi Since $P$ is smooth, $A^p(P) \cap A_q(P) \subset A_{q-p}(P)$ by \cite[8.3 (b)]{Ful}\comment{p.140}.  As $\text{Td}(P) = [P] + r_P$ with $r_P \in A_{<n}(P)$, by equating terms in each degree, we find $\text{ch}(i_\ast {\cal{O}}_V) \in A^{\geq n-a}(P)$.  

\medskip \noi By Grothendieck-Riemann-Roch (for the smooth $P$, as in \cite[15.2.1]{Ful}) and the action of ch on $\otimes$, $\chi ({\cal{F}} \dotimes {\cal{G}}) = \int_P \text{ch}({\cal{F}}) \cdot \text{ch}({\cal{G}}) \cdot \text{Td}_P$.  Here $\cdot$ means intersection product of cycle classes.  The first possible nonzero term in $\text{ch}({\cal{F}}) \cdot \text{ch}({\cal{G}})$ would come from $\text{ch}_{n-a}({\cal{F}}) \cdot \text{ch}_{n-b}({\cal{G}})$, but this term is zero for degree reasons.
\end{proof}

\begin{proposition} \label{GRR move} Let $T$ be the spectrum of a field, and $Z \into P_T$ an $a$-dimensional subscheme.  Let $M,N \in \coh(P_T)$ be invertible sheaves on some subvariety of $P_T$ of dimension $\leq n-a$, and suppose we have exact sequences:
$$0 \to M \xrightarrow{s} N \to Q_s \to 0$$
$$0 \to M \xrightarrow{t}  N \to Q_t \to 0 $$
such that $Z \cap \supp(Q_s) = Z \cap \supp(Q_t) = \emptyset$.

Then the unique $a_Z \in \Gamma(T, \cal{O}^*_T)$ making the following diagram commute:
$$\xym{ f_T(\cal{O}_Z,Q_s) \ar[r]^-{\varphi^Z_T}  \ar[d]_-{\text{via }s} & \cal{O}_T \ar[dd]^-{a_Z} \\
{(f_T(\cal{O}_Z, M))}^{-1} \otimes f_T (\cal{O}_Z,N) \\
f_T(\cal{O}_Z,Q_t) \ar[u]^-{\text{via }t} \ar[r]^-{\varphi^Z_T} & \cal{O}_T \\ }$$
depends only on $[Z]$, i.e., $a_Z = a_{Z'}$ if $[Z] = [Z']$.

\end{proposition}

\begin{proof} To prove the claim it is equivalent to show that the difference between $f(1 \otimes s), f(1 \otimes t) : f_T(\cal{O}_Z,M) \cong f_T(\cal{O}_Z,N)$ depends only on $[Z]$.

\textbf{Step 1.}  By taking a filtration of $\cal{O}_Z$ such that the graded pieces are isomorphic to (twists of) structure sheaves of subvarieties, and using the additivity of the determinant on filtrations, we are reduced to showing that if $\cal{F} \in \coh(P_T)$ with $\supp(\cal{F}) \subset \supp(Z)$ and $\dim(\supp(\cal{F})) \leq a-1$, the induced isomorphisms $f(1\otimes s), f(1 \otimes t) : f(\cal{F}, M) \cong f(\cal{F}, N)$ are equal.  (Such a filtration exists by \cite[I.7.4]{H}\comment{[H] Ch. 1 Prop 7.4 + tilde is exact [H] Ch.2 p.120ish}.)  The subquotients of the filtration of $\cal{O}_Z$ depend on the filtration chosen, but the top-dimensional components always appear with their correct multiplicities, i.e., the cycle $[Z]$ can be extracted from the filtration.

\textbf{Step 2.} Let $Q_U$ denote the cokernel of the universal $\cal{O}_P$-homomorphism $M \to N$; note $Q_U$ is flat over $\Hom_{\cal{O}_P}(M,N) \setminus 0$ since a morphism between invertible sheaves is either injective or zero, hence the Euler characteristic of every cokernel is $\chi(M) - \chi(N)$.  We consider the line bundle $\det \textbf{R} \pi_\ast (p^*_1 (\cal{F}) \dotimes Q_U)$ on $\Hom_{\cal{O}_P}(M,N) \setminus 0$.  Its fiber over $s \in \Hom_{\cal{O}_P}(M,N)$ is precisely $f(\cal{F}, Q_s)$.  Since $Q_s = Q_{\lambda s}$ for $\lambda \in \Gamma(T, \cal{O}^*_T)$, we consider $\det \textbf{R} \pi_\ast (p^*_1(\cal{F}) \dotimes Q_U)$ as a line bundle on the projective space $\bb{P}(\Hom_{\cal{O}_P}(M,N) \setminus 0)$.

We claim this line bundle is trivial.  To prove this it suffices to show it is trivial along a line $\bb{P}^1 \cong L \into \bb{P}(\Hom_{\cal{O}_P}(M,N) \setminus 0)$.  For this purpose Grothendieck-Riemann-Roch (i.e., ignoring torsion) is adequate.  More precisely, we consider the GRR diagram:

$$\xym{
K_0(P \times L) \ar[d]^-{\textbf{R} \pi_\ast} \ar[rrr]^-{\ch(-) \cdot \Td(P) \cdot \Td(L)} &&& {A_*(P \times L)}_\bb{Q} \ar[d]^-{\pi_\ast} \\
K_0(L) \ar[rrr]^-{\ch(-) \cdot \Td(L)} \ar[d]^-\det &&& {A_*(L)}_\bb{Q} \\
\pic(L) \ar[urrr]^-{c_1} \\ }$$

We have $\dim(\supp(p^*_1(\cal{F}))) \leq a, \dim(\supp(Q_U)) \leq n-a$, and $\dim(P \times L) = n+1$.  Hence $\ch (p^*_1 (\cal{F})) \cdot \ch (Q_U) = 0$ in ${A_*(P \times L)}_\bb{Q}$ for degree reasons (as in the proof of $\ref{chisubsub}$), so $p^*_1 (\cal{F}) \dotimes Q_U \in K_0(P \times L)$ maps to $0$ in the top row.  Hence $c_1(\det \textbf{R} \pi_\ast (p^*_1(\cal{F}) \dotimes Q_U))$ is a torsion class, and therefore $\det \textbf{R} \pi_\ast (p^*_1(\cal{F}) \dotimes Q_U)$ is trivial.

\textbf{Step 3.} For $s \in \Hom_{\cal{O}_P}(M,N) \setminus 0$, consider the induced identification $f(s): f(\cal{F},Q_s) \otimes f(\cal{F}, M) \cong f(\cal{F},N)$.  Since $f(\cal{F},Q_s)$ is canonically trivial, i.e., the trivialization induced by $\supp (\cal{F}) \cap \supp(Q_s) = \emptyset$ extends over all $\Hom_{\cal{O}_P}(M,N) \setminus 0$, we consider $f(s)$ as an isomorphism $f(\cal{F},M) \cong f(\cal{F},N)$.  Since $\chi (\cal{F} \dotimes M) = \chi (\cal{F} \dotimes N) = 0$ by $\ref{chisubsub}$, we have $f(s) = f(\lambda s)$ for $\lambda \in \Gamma(T,\cal{O}^*_T)$.  Therefore we have a commutative diagram:

$$\xym{ \Hom(M,N) \setminus 0 \ar[d] \ar[r]^-{s \mapsto f(s)} & \text{Isom}(f(\cal{F},M), f(\cal{F},N)) \\
\bb{P}(\Hom(M,N) \setminus 0) \ar[ru] \\ }$$

But there are no nonconstant functions on $\bb{P}(\Hom(M,N) \setminus 0)$, hence $f(s) = f(t) : f(\cal{F},M) \cong f(\cal{F},N)$.

\end{proof}

\subsection{Further properties of $D_{Z,W}$.}

Assuming the base and incidence are optimal, we can already prove the following important property of $D_{Z,W}$.

\begin{proposition} \label{varphi proper HC} Let $(Z,Z',W)$ be a Hilbert-Chow datum over a normal base $T$, and suppose the incidence $Z \cap W$ satisfies:
\begin{enumerate}
\item $(Z \cap W)_\eta  = \emptyset$ for all generic points $\eta \in T$; and
\item $Z \cap W$ is finite over all points of depth 1 in $T$.
\end{enumerate}
Then the Cartier divisors $D_{Z,W}, D_{Z',W}$ are equal. \end{proposition}

\begin{proof} Since the smooth locus of $T$ contains all points of depth 1, and the formation of $D_{Z, W}$ is compatible with the inclusion $T^{sm} \subset T$, we may assume $T$ is smooth.  For $T$ regular there is a canonical isomorphism \cite[Prop.~8]{KM}:

$$f_T(\cal{O}_Z, \cal{O}_W) \cong \otimes_{p,q} {(\text{det}_T R^q \pi_\ast (\cal{H}^p(\cal{O}_Z \dotimes \cal{O}_W)))}^{{(-1)}^{p+q}}.$$

To calculate the coefficient of a depth 1 point $t \in T$ in $D_{Z,W}$, we may replace $T$ with the spectrum of the DVR $\cal{O}_{T,t}$.  Then the support of $\cal{H}^p(\cal{O}_Z \dotimes \cal{O}_W)$ is finite over $T$ (indeed, over $t$), so in the displayed expression only the terms with $q=0$ can contribute.  

By \cite[Thm.3(vi)]{KM}, the multiplicity of a depth one point is determined by the sum $\sum_i {(-1)}^i \ell_t( \cal{H}^i(\textbf{R} \pi_\ast (\cal{O}_Z \dotimes \cal{O}_W)))$.  This last sum is equal to
$$\sum_i {(-1)}^i \ell_t( \pi_\ast \cal{H}^i(\cal{O}_Z \dotimes \cal{O}_W) )= (\deg \pi) ( \sum_{p, t' \to t} {(-1)}^p \ell_{t'} (  \cal{H}^p(\cal{O}_Z \dotimes \cal{O}_W)  )) ,$$
hence it suffices to show $\gamma(\cal{O}_Z) := \sum_{p, t' \to t} {(-1)}^p \ell_{t'} ( \cal{H}^p(\cal{O}_Z \dotimes \cal{O}_W ))$ depends only the underlying cycle $[Z]$.  We remark that if $Z$ and $W$ are integral and $b=n-a-1$, then the contribution of a point $t'$ (lying over $t$) to $\gamma(\cal{O}_Z)$ is exactly Serre's Tor-formula for the intersection index of $Z$ and $W$ at $t'$ \cite[V.C.Thm.1(b)]{Ser}.

Without loss of generality we assume $W$ is integral and $\dim(W) =b=n-a-1$; we will see in the proof all sums are 0 if $b < n-a-1$.  Given an exact sequence $0 \to \cal{F}_1 \to \cal{F}_2 \to \cal{F}_3 \to 0$ of coherent sheaves on $P_T$ with support of relative dimension $\leq a$ and satisfying the incidence hypothesis with respect to $W$, by the long exact cohomology sequence we obtain $\gamma (\cal{F}_1) + \gamma (\cal{F}_3) = \gamma(\cal{F}_2)$.  It then follows $\gamma$ is additive on filtrations.

Write $[Z] = \sum_i a_i Z_i$.  Again by \cite[I.7.4]{H}\comment{[H] Ch. 1 Prop 7.4 + tilde is exact [H] Ch.2 p.120ish}, the sheaf $\cal{O}_Z$ admits a filtration whose subquotients are invertible sheaves $L_i$ on subvarieties contained in $Z$; and each top-dimensional component $Z_i$ appears exactly $a_i$ times.  Since the invertible sheaves are twists by an ample class on $P_T$, we may assume there is either an injective map $\cal{O}_{Z_i} \to L_i$ or an injective map $L_i \to \cal{O}_{Z_i}$.  Therefore $\gamma(\cal{O}_Z) = \sum_i \gamma(L_i) = \sum_i a_i \gamma (\cal{O}_{Z_i})$ modulo summands of the form $\gamma(\cal{F})$ where $\cal{F}$ is a sheaf on $P_T$ whose support over the generic point of $T$ has dimension $\leq a-1$.  So it suffices to show $\gamma$ vanishes on sheaves of this type.  Again we may assume $\cal{F}$ is isomorphic to the structure sheaf of a subvariety $Y \into Z$ in $P_T$.  Note that $\dim Y \leq a$, else, being contained in $Z$, $Y$ would dominate $T$ and have all fibers of dimension $\geq a$; and then $Y$ would contribute to the cycle $[Z]$ of $Z$.

There are two cases to consider: if $Y \cap W = \emptyset$, then $\gamma(\cal{F}) =0$ since $\cal{F} \dotimes \cal{O}_W$ is acyclic.  If $Y \cap W \neq \emptyset$, then the intersection $Y \cap W$ is improper, and the Tor-formula vanishes at components of improper intersection (due to Serre in the equal characteristic case \cite[V.C.Thm.1(a)]{Ser}).
\end{proof}

\begin{remarks}
\subtheorem The proof shows the relation between the incidence line bundle and Serre's Tor-formula for intersection multiplicities; and also that our construction agrees with Mazur's at least on the normalization of the locus $U'$.  Thus the essential tasks are to extend the divisor through the locus where the expected incidence condition (the hypothesis in $\ref{varphi proper HC}$) fails, and to remove the assumption of normality.
\subtheorem The assumption (C2) of \cite{BK} is that the incidence is generically finite over, and nowhere dense in, its image in the base $S$.  These assumptions imply the map $S \to \scr{C}_a(P) \times \scr{C}_{n-a-1}(P)$ factors through $U'$.
\subtheorem Considering the GRR diagram as in the proof of $\ref{GRR move}$ with $L$ replaced by a general regular base $T$, one sees that the first Chern class of the incidence line bundle $f_T(\cal{O}_Z, \cal{O}_W)$ modulo torsion depends only on the underlying cycles $[Z], [W]$, independent of any assumption of properness of intersection.  By contrast in $\ref{varphi proper HC}$ we have the result integrally.
 \subtheorem We may write $D_{[Z],W}$ in the case we have a Hilbert datum as in $\ref{varphi proper HC}$, e.g., with proper intersection over a regular base $T$.
\end{remarks}

\begin{corollary} \label{phi proper} Among Hilbert-Chow data:
\begin{enumerate}
\item over normal bases $T$; and
\item such that $Z,W$ (and hence $Z',W$) are generically disjoint and have finite incidence over points of depth 1 in $T$;
\end{enumerate}
there exists a collection of isomorphisms $\phi_T^{Z,Z'} : f_T(\cal{O}_Z, \cal{O}_W) \cong f_T(\cal{O}_{Z'}, \cal{O}_W)$ which:
\begin{enumerate}
\item is compatible with base change preserving the incidence condition (2);
\item satisfies the cocycle condition; and
\item agrees with the collection on disjoint families defined in $\ref{phi disjoint}$.
\end{enumerate}
 \end{corollary}
\begin{proof}
We define $\phi_T^{Z,Z'} := {(\varphi_T^{Z'})}^{-1} \circ \varphi_T^Z : f_T(\cal{O}_Z, \cal{O}_W) \cong \cal{O}_T(D_{[Z],W}) \cong f_T(\cal{O}_{Z'}, \cal{O}_W)$.  \end{proof}

\begin{construction-notation}\label{general D} We continue with the Cartier divisor $D_{Z,W} \into T$ associated to a Hilbert datum ($\ref{varphi proper H, D adds}$).  For $Z \into P_T$ a $T$-flat family of $a$-dimensional subschemes of $P$ and $s \in \bb{Z}_{\geq 0} $, let $\coh_{\leq s; Z_T}(P)$ denote the abelian category of coherent sheaves $\cal{G}$ on $P$ such that $\dim (\supp(\cal{G})) \leq s$ and $(Z \cap \supp(\cal{G}))_\eta = \emptyset$ for all generic points $\eta \in T$.  For $\cal{G} \in \coh_{\leq n-a-1; Z_T}(P)$, we obtain a Cartier divisor $D_{Z, \cal{G}} \into T$ and a canonical isomorphism $f_T (\cal{O}_Z, \cal{G}) \cong \cal{O}_T (D_{Z, \cal{G}})$.

We denote by $K_0^{s; Z}(P)$ the $K_0$-group of the abelian category $\coh_{\leq s; Z_T}(P)$: we take the free abelian group on sheaves in $\coh_{\leq s; Z_T}(P)$, then impose relations from short exact sequences whose terms all lie in $\coh_{\leq s; Z_T}(P)$.  Let $\cdiv(T)$ denote the group of Cartier divisors on $T$.
\end{construction-notation}

We summarize some elementary properties of this construction in the following proposition.

\begin{proposition} \label{D elem} \begin{enumerate}
\item The map $D_{Z, -}: \coh_{\leq n-a-1; Z_T}(P) \to \cdiv(T)$ defined by $\cal{G} \mapsto D_{Z, \cal{G}}$ descends to a homomorphism $K_0^{n-a-1; Z}(P) \to \cdiv(T)$ (which we also denote by $D_{Z, -}$).  
\item If $f: S \to T$ is a morphism and $\cal{G} \in \coh_{\leq s; Z_S}(P)$, then $f^*(D_{Z, \cal{G}}) = D_{Z, \cal{G}}$ as divisors on $S$.
\item If $\cal{G} \in \coh_{\leq n-a-1; Z_T}(P)$ satisfies $Z \cap \cal{G} = \emptyset$, then $D_{Z, \cal{G}} = 0$.
\item If $\cal{G} \in  \coh_{\leq 0; Z_T}(P)$ and the divisorial part of the family $Z \into P_T$ is trivial (i.e., $a \leq \dim(P) -2$), then $D_{Z, \cal{G}} = 0$.
\end{enumerate}
\end{proposition}

\begin{proof} The first three properties follow immediately from the additivity, compatibility with base change, and compatibility with the trivialization of an acyclic complex, of the associated divisor construction discussed in Section $\ref{Kprelim}$.  To prove the last property, by the first property we may assume $\cal{G}$ is the structure sheaf of a zero-dimensional subvariety $W \into P$.  (If $k = \overline{k}$, this is just a single closed point, but we give an argument here valid for any $T$-flat family $W \into P_T$ of zero-dimensional subschemes, such that $W$ is integral.)

By \cite[5.3]{ZD} there is a canonical isomorphism:
$$\text{det}_T \textbf{R} \pi_\ast (\cal{O}_Z \dotimes \cal{G}) \cong {(\text{det}_T \pi_\ast \cal{G})}^{\text{rk}(\cal{O}_Z) -1} \otimes \text{det}_T (\pi_\ast (\text{det}_{P_T}(\cal{O}_Z) \arrowvert_W )).$$
Since the divisorial part of $\cal{O}_Z$ was assumed to be empty, the line bundle $\det_{P_T} (\cal{O}_Z)$ is canonically trivial; and $\text{rk}(\cal{O}_Z) -1 = -1$.  Therefore the right hand side is canonically trivial, so $D_{Z, \cal{G}} = 0$.
\end{proof}

\begin{remark} In case $a = \dim(P)-1$, we refer to \cite{ZD}.  The reduction to the diagonal shows we may assume $\dim(P)=1$ or $a \leq \dim(P)-2$; but to the extent we rely on (4) in $\ref{D elem}$, we must understand the case of zero-cycles and divisors. \end{remark}

\begin{proposition} \label{1dim base no cartier} Suppose $T$ has dimension $\leq 1$ and $\cal{G} \in \coh_{\leq n-a-2; Z_T}(P)$.  Then $D_{Z, \cal{G}} = 0$. \end{proposition}

\begin{remark} The condition $D_{Z, \cal{G}} = 0$ is equivalent to the canonical trivialization (induced by the generic acyclicity of $\textbf{R} \pi_\ast (\cal{O}_Z \dotimes \cal{G})$) extending over all of $T$. \end{remark}

\begin{proof} We may also assume $\cal{G}$ is the structure sheaf of a subvariety $W \into P$ of dimension $b \leq n-a-2$, as these sheaves generate the $K_0$-group.

Since $D_{Z, \cal{G}} = 0$ for $\dim( \supp (\cal{G})) \leq 0$ (again by $\ref{D elem}$), it suffices to prove the following claim: if $D_{Z, \cal{G}} = 0$ for all $\cal{G} \in \coh_{\leq b-1; Z}(P)$, and $1 \leq b \leq n-a-2$, then $D_{Z, \cal{G}} = 0$ for all $\cal{G} \in \coh_{\leq b; Z}(P)$.

We prove this claim: by $\ref{moving lemma 2}$ there are short exact sequences:
$$0 \to M_i \xrightarrow{s^i_0} \cal{O}_{B_i} \to Q^i_0 \to 0$$
$$0 \to M_i \xrightarrow{s^i_\infty} \cal{O}_{B_i} \to Q^i_\infty \to 0$$
in $\coh_{\leq n-a-1; Z_T}(P)$, such that the $b$-dimensional cycle $W + \sum_i ([Q^i_0] - [Q^i_\infty])$ is disjoint from $Z$.

If $[Q] = \sum_j a_j W_j$, then we write $D_{Z, [Q]} := \sum_j a_j D_{Z, W_j}$.  The short exact sequences give $\sum_i D_{Z, Q^i_0} = \sum_i D_{Z, Q^i_\infty}$ in $K_0^{n-a-1; Z}(P)$.  Furthermore the difference $Q^i_* - [Q^i_*]$ lies in $F_{b-1}(K_0(P))$.

But then

\sss

$D_{Z, W} = D_{Z, W} + \sum_i (D_{Z, Q^i_0} - D_{Z, Q^i_\infty} ) $ from the SES

\sss

$= D_{Z,W} + \sum_i (D_{Z, [Q^i_0]} - D_{Z, [Q^i_\infty]})$ since we assumed $D_{Z,-}$ vanishes on $F_{b-1}(K_0(P))$

\sss

$= D_{Z, W + \sum_i ([Q^i_0] -  [Q^i_\infty])} = 0$ by the disjointness from $Z$.
\end{proof}

\begin{corollary} Suppose $T$ is normal and $\cal{G} \in \coh_{\leq n-a-2; Z_T}(P)$.  Then $D_{Z, \cal{G}} = 0$. \end{corollary}

\begin{proof} For $t \in T$ of depth 1, the formation of $D_{Z, \cal{G}}$ is compatible with the morphism $g_t: \spec \cal{O}_{T,t} \to T$.  By $\ref{1dim base no cartier}$, we have $g_t^*(D_{Z, \cal{G}}) = 0$ for all such $t$.  Since $D_{Z, \cal{G}}$ is determined its restriction to points of depth 1, the result follows. \end{proof} 

\begin{proposition} Suppose $T$ is seminormal and $\cal{G} \in \coh_{\leq n-a-2; Z_T}(P)$.  Then $D_{Z, \cal{G}} = 0$. \end{proposition}

\begin{proof} The formation of the divisor $D_{Z, \cal{G}}$ is compatible with the (finite, birational) normalization morphism $\nu : T^\nu \to T$.  By the previous result we know local equations for $\nu^*(D_{Z, \cal{G}})$ are units.  We need to show these units are constant along the fibers of $\nu$.  Suppose $t \in T$ has branches $b_1, \ldots, b_r$ in $T^\nu$ .  For each $b_i$ there exists a DVR $R_i$ and a morphism $g_i : \spec R_i \to T$ such that: $R_i$ has residue field $\kappa(t)$, and $g_i$ covers a generization of $t$ to the locus of subschemes disjoint from $\supp(\cal{G})$.  Denote by $S$ the union of the $\spec R_i$s glued along $\spec \kappa(t)$.  Then we have a morphism $g: S \to T$, and by $\ref{1dim base no cartier}$  we conclude $g^*(D_{Z, \cal{G}}) = 0$.  The corresponding trivialization $f_S(\cal{O}_Z, \cal{G}) \cong \cal{O}_S$ is our candidate for extending the trivialization through $t$. 

Let $\text{Exc}(\nu) \into T$ denote the locus over which $\nu$ is not an isomorphism, and let $I_{\cal{G}} \into T$ denote the locus of subschemes $Z$ such that $Z \cap \supp(\cal{G}) \neq \emptyset$.  Set $U := T - (\text{Exc}(\nu) \cap I_{\cal{G}} )$.  Then we have a trivialization $\varphi_U: f_U(\cal{O}_Z, \cal{G}) \cong \cal{O}_U$, and this extends to $\varphi_{T^\nu} : f_{T^\nu}(\cal{O}_Z, \cal{G}) \cong \cal{O}_{T^\nu}$.  For $t \not \in U$, we constructed in the previous paragraph the isomorphism $t_S: f_S(\cal{O}_Z, \cal{G}) \cong \cal{O}_S$, which we may restrict to $t$.  Together these define a pointwise trivialization $f_T(\cal{O}_Z, \cal{G}) \cong \cal{O}_T$, i.e., a nonzero element in every fiber of the line bundle $f_T(\cal{O}_Z, \cal{G})$.  We form the cartesian diagram:
$$\xym{S^\nu \ar[r]^-{g^\nu} \ar[d]_-{\nu |_S} & T^\nu \ar[d]^-\nu \\ S \ar[r]^-g & T \\}$$
Since the formation of $D_{Z, \cal{G}}$ is compatible with all of the morphisms appearing in this diagram, we obtain ${(\nu |_S)}^*(t_S) = {(g^\nu)}^*(\varphi_{T^\nu})$.  It then follows $\varphi_{T^\nu}$ is constant along the fibers of $\nu$, hence it descends to a trivialization $\varphi_T : f_T(\cal{O}_Z, \cal{G}) \cong \cal{O}_T$. 
\end{proof}

\begin{corollary} Suppose $T$ is seminormal and $\cal{G} \in \coh_{\leq n-a-1; Z_T}(P)$.  Then $D_{Z, \cal{G}} = D_{Z, [\cal{G}]}$. \end{corollary}

\begin{remark} For $T$ smooth and $\cal{G} \in \coh_{\leq n-a-2; Z_T}(P)$, one can deduce the line bundle $f_T (\cal{O}_Z, \cal{G})$ is trivial as follows.  The filtration of the $K_0(X)$-group by dimension of support is compatible with multiplication, if $X$ is a smooth quasi-projective scheme over a field \cite[Exp.0 Ch.2 Sect.4 Thm.2.12 Cor.1]{SGA6}.  From $\cal{O}_Z \in F_{a+ \dim T}(K_0(P_T))$ and $\cal{G} \in F_{n-a-2 + \dim T}(K_0(P_T))$ it follows that $\cal{O}_Z \dotimes \cal{G} \in F_{\dim T -2} (K_0(P_T))$.  Therefore $\textbf{R} \pi_\ast (\cal{O}_Z \dotimes \cal{G}) \in F_{\dim T -2}(K_0(T))$, and hence $f_T(\cal{O}_Z, \cal{G}) \cong \cal{O}_T$.

For a general base $T$, this reasoning is valid rationally, hence we can conclude $f_T(\cal{O}_Z, \cal{G}) $ is a torsion line bundle.

Since the dimension filtration's compatibility with multiplication can be viewed as a consequence of the moving lemma, in some sense we have given this proof.  Because we need to keep track of the trivialization and not just the abstract invertible sheaf, we work with Cartier divisors rather than line bundles.
\end{remark}

\begin{corollary} \label{cartier}  Let $(Z,W)$ be a Hilbert datum over any base $T$ such that $(Z \cap W)_\eta = \emptyset$ for every generic point $\eta \in T$.  Suppose further that $W = W_k \times_k T$ for a $k$-scheme $W_k$ with $[W_k]=\sum_i a_i W_i$.  Then there is a canonical isomorphism $f_T(\cal{O}_Z, \cal{O}_W) \cong \otimes_i {(f_T(\cal{O}_Z, \cal{O}_{W_i}))}^{\otimes a_i}$. \end{corollary}

\subsection{Application to the incidence line bundle}

\begin{construction} The facts $\ref{moving lemma}$ and $\ref{cartier}$ produce a construction; in the description we suppress the base $T$ and we use $f(-) := f(\cal{O}_Z, -)$ as before, since the first factor is constant.  We use the identification of $\ref{cartier}$: $f(\cal{O}_W) \cong \otimes_i {(f(\cal{O}_{W_i}))}^{\otimes a_i} =: f([\cal{O}_W])$.  (In our situation $W$ is a Cartier divisor on a $(b+1)-$dimensional subvariety $B \into P$ with $[W] = \sum_i a_i W_i$.)

\smallskip \noi Now given $(Z,W); B_i, M_i, s^i_{0, \infty}$ as in $\ref{moving lemma}$, let $f(s_*) : f(M) \otimes f(Q_*) \cong f(\cal{O}_B)$ denote the isomorphism induced by the short exact sequence.  We set
$$\alpha := [W]+ \sum_i ([Q^i_0] - [Q^i_\infty])$$
to be the moved $b$-dimensional cycle.  Then we have a canonical isomorphism $f([\cal{O}_W]) \otimes (\otimes_i f([Q^i_0])) = f(\alpha) \otimes (\otimes_i f([Q^i_\infty]))$.

\smallskip \noi Let $\beta^Z_{W,\alpha} : f([\cal{O}_W]) \cong f(\alpha)$ be the unique isomorphism making the following diagram commute:
$$ \xym{ f([\cal{O}_W]) \otimes (\otimes_i (f([Q^i_0]) \otimes f(M_i))) \ar[d]^-{1 \otimes (\otimes_i f(s^i_0))} \ar[r]^-= &   f(\alpha) \otimes (\otimes_i (f([Q^i_\infty]) \otimes f(M_i))) \ar[d]^-{1 \otimes (\otimes_i f(s^i_\infty))}  \\
f([\cal{O}_W]) \otimes (\otimes_i f(\cal{O}_{B_i})) \ar[r]^-{\beta^Z_{W,\alpha}  \otimes 1} & f(\alpha) \otimes (\otimes_i f(\cal{O}_{B_i})) \\ }$$

In the top row, we tensored the canonical isomorphism with the identity on the $f(M_i)$ factors.  
\end{construction}

This construction extends the collection $\{ \phi_T \}$ to spectra of local rings $T$.

\begin{proposition} \label{proper compat moving} Let $(Z,Z',W)$ be a Hilbert-Chow datum over $T$ the spectrum of a local ring with $W = W_k \times_k T$, and suppose $(Z \cap W)_\eta =  (Z' \cap W )_\eta = \emptyset$ for all generic points $\eta \in T$.  Suppose subvarieties $B_1, \ldots , B_n \subset P$ and short exact sequences as in $\ref{moving lemma}$ have been chosen.  Then the isomorphism:
$$ {(\varphi^{Z'; \alpha} \circ \beta^{Z'}_{W,\alpha})}^{-1} \circ (\varphi^{Z;\alpha} \circ \beta^Z_{W,\alpha}) : f_T(\cal{O}_Z, \cal{O}_W) \cong f_T(\cal{O}_Z, \alpha) \cong \cal{O}_T \cong f_T(\cal{O}_{Z'},\alpha) \cong f_T(\cal{O}_{Z'}, \cal{O}_W)$$
is independent of the choice of move: given another collection of data $(\hat{B_i}, \hat{M_i}, \hat{s^i}_{0, \infty})$ producing a $b$-dimensional cycle $\hat{\alpha}$ (also disjoint from $Z$), we have:
$${(\varphi^{Z'; \alpha} \circ \beta^{Z'}_{W,\alpha} )}^{-1} \circ (\varphi^{Z;\alpha} \circ \beta^Z_{W,\alpha}) = {(\varphi^{Z'; \hat{\alpha}} \circ \beta^{Z'}_{W,\hat{\alpha}})}^{-1} \circ ( \varphi^{Z; \hat{\alpha}} \circ \beta^Z_{W,\hat{\alpha}} ).$$

Furthermore ${(\varphi^{Z'; \alpha} \circ \beta^{Z'}_{W,\alpha})}^{-1} \circ  (\varphi^{Z;\alpha} \circ \beta^Z_{W,\alpha})$ agrees with the canonical identifications at every generic $\eta \in T$.  If in addition $T$ is the spectrum of a regular local ring and the incidence $Z \cap W$ satisfies the hypotheses of  $\ref{phi proper}$, then $\phi_T^{Z,Z'} = {(\varphi^{Z'; \alpha} \circ \beta^{Z'}_{W,\alpha})}^{-1} \circ (\varphi^{Z;\alpha} \circ \beta^Z_{W,\alpha}).$
\end{proposition}
\begin{proof} We claim any choice of moving data produces the canonical isomorphism over the generic points of $T$.  Now $\ref{GRR move}$ shows the choice of short exact sequences affects the map $\varphi^{Z;\alpha}_\eta \circ {(\beta^Z_{W,\alpha})}_\eta : f_\eta(\cal{O}_Z,\cal{O}_W) \cong f_\eta(\cal{O}_Z,\alpha) \cong \cal{O}_\eta$ by a constant depending only on $[Z]$, and this is exactly canceled by the map back to $ f_T(\cal{O}_{Z'}, \cal{O}_W)$.  In other words, since $a_Z = a_{Z'}$, the following diagram commutes.

$$\xym{f_\eta(\cal{O}_Z, \cal{O}_W) \ar[d]_-{{(\beta^Z_{W,\alpha})}_\eta} \ar[r]^-{\varphi^{Z;W}_\eta} \ar@/^2pc/[rrr]^-{\phi_\eta^{Z,Z'}}  & \cal{O}_{T,\eta} \ar[r]^-= \ar[d]_-{a_Z}& \cal{O}_{T,\eta} \ar[d]^-{a_{Z'}}& f_\eta(\cal{O}_{Z'},\cal{O}_W) \ar[d]^-{{(\beta^{Z'}_{W,\alpha})}_\eta} \ar[l]_-{\varphi^{Z';W}_\eta} \\
f_\eta(\cal{O}_Z,\alpha) \ar[r]_-{\varphi_\eta^{Z; \alpha}} & \cal{O}_{T,\eta} \ar[r]_-= &\cal{O}_{T,\eta} & f_\eta(\cal{O}_{Z'},\alpha) \ar[l]^-{\varphi_\eta^{Z'; \alpha}}  \\ }$$

From this our claims follow: two choices of moving data produce isomorphisms which agree at the generic points of $T$, hence they agree; and the isomorphism $\phi_T^{Z,Z'} $ for $T$ regular is characterized by agreeing with the composition of the canonical trivializations over the generic points of $T$.
\end{proof}

\begin{corollary} \label{phi generic disjoint} Among Hilbert-Chow data over bases $T$ satisfying:
\begin{enumerate}
\item $T$ is either regular, or local; and
\item all generic points of $T$ correspond to disjoint subschemes;
\end{enumerate}
there exists a collection of isomorphisms $\{ \phi_T \}$ which:
 \begin{enumerate}
 \item is compatible with base change preserving generic disjointness;
 \item satisfies the cocycle condition; and
 \item extends the collection of $\ref{phi proper}$.
 \end{enumerate}
\end{corollary}
\begin{proof}  To see our construction commutes with base change preserving generic disjointness, note that moves valid over $T$ (i.e., producing $\beta_{W,\alpha}$) pullback via $S \to T$ to suitable moves on $S$.  The cocycle condition (an equality of two isomorphisms of line bundles on a reduced scheme) holds at the generic points, hence it holds everywhere. 
\end{proof}

The previous corollary provides us an isomorphism $f_s(\cal{O}_Z, \cal{O}_W) \cong f_s(\cal{O}_{Z'}, \cal{O}_W)$ for a Hilbert-Chow datum over $s$ the spectrum of a field corresponding to a point in the incidence locus, namely the restriction of the isomorphism over the local ring, possibly followed by a field extension.  In the next proposition we observe this is compatible with specializations from the locus of disjoint subschemes into the incidence locus.

\begin{proposition} Let $(Z,Z',W)$ be a Hilbert-Chow datum over $s$ the spectrum of a field $\kappa(s)$ corresponding to a point of incidence, i.e., $Z \cap W, Z' \cap W \neq \emptyset$.  Let $(Z_{T}, Z'_{T}, W)$ be a Hilbert-Chow datum over $T$ the spectrum of a DVR covering a generization from $s$ to the locus of disjoint subschemes, and with $T_0 = s$.  Then the isomorphism
$$ f_s(\cal{O}_Z, \cal{O}_W) \xrightarrow{\text{can}} f_{T}(\cal{O}_{Z_{T}}, \cal{O}_W ) \times_T s \xrightarrow{(\phi_T^{Z_T, Z'_T}) \times_T s} f_T(\cal{O}_{Z'_T}, \cal{O}_W ) \times_T s \xleftarrow{\text{can}} f_s(\cal{O}_{Z'}, \cal{O}_W)$$
is equal to the isomorphism induced by $\phi_R := \phi_{\spec R}$, where $(R,\mf{m}, K= R /\mf{m})$ is the (seminormal) local ring of the image of $s$ on ${(\scr{H}_a \times_{\scr{C}_a} \scr{H}_a)}^{sn}$.  In other words, the previously displayed isomorphism is equal to $(-) \otimes_K \kappa(s)$ of the following isomorphism:
$$ f_K (\cal{O}_Z, \cal{O}_W) \xrightarrow{\text{can}} f_{R}(\cal{O}_{Z_{R}}, \cal{O}_W ) \otimes_R K \xrightarrow{(\phi_R^{Z_R, Z'_R}) \otimes_R K} f_R(\cal{O}_{Z'_R}, \cal{O}_W ) \otimes_R K \xleftarrow{\text{can}} f_{K}(\cal{O}_{Z'}, \cal{O}_W) .$$
In particular, if we generize to the locus of disjoint subschemes and then restrict, the resulting isomorphism at $s$ is independent of the choice of generization.
\end{proposition}
\begin{proof} 
Since the collection $\{ \phi_T \}$ is compatible with base change preserving generic disjointness, we may replace $R$ by $R/I$ for some ideal $I \subset R$ such that $R/I$ is a local domain whose generic point $\spec L$ is the image of the generic point of $T$.  Therefore we have commutative squares:

$$\xym{ L \ar[r] & \kappa(\eta) \\ R \ar[u] \ar[d] \ar[r] & \Gamma(T,\cal{O}_T) \ar[u] \ar[d] \\ K \ar[r] & \kappa(s) \\ }$$
To show $\phi_R \times_R  T = \phi_T$, it suffices to show they agree at $\eta$, i.e., that $(\phi_R \times_R  T) \times_T \eta = \phi_T \times_T \eta$.  But this is equivalent to $(\phi_L) \times_L \eta = \phi_\eta$, which is a consequence of the compatibility with base change on pairs of disjoint subschemes ($\ref{phi disjoint}$).
\end{proof}

\begin{corollary} \label{near existence DD} Among Hilbert-Chow data satisfying at least one of the following conditions:
\begin{enumerate}
\item the conditions of $\ref{phi generic disjoint}$;
\item the base $T$ is a field;
\end{enumerate}
there exists a collection of isomorphisms $\{ \phi_T \}$ which:
\begin{enumerate}
\item satisfies the cocycle condition; 
\item extends the collection of $\ref{phi generic disjoint}$ (so is compatible with base change preserving generic disjointness); and 
\item is compatible with specialization from the locus of disjoint subschemes to the incidence locus.
\end{enumerate} 
\end{corollary}
\begin{proof}The compatibility with specialization is built into the construction.  The new feature to check is the cocycle condition on field points mapping to the incidence locus.  But if the cocycle condition holds after generization and the covering isomorphism is compatible with base change, then the collection must also satisfy the cocycle condition at new (field) points.  
\end{proof}

We augment $\ref{near existence DD}$ to include specializations fully within the incidence locus, hence we have the first part of $\ref{main thm intro}$.

\begin{theorem} \label{existence of DD} Among Hilbert-Chow data satisfying at least one of the following conditions:
\begin{enumerate}
\item the conditions of $\ref{phi generic disjoint}$;
\item the base $T$ is a field;
\item the base $T$ is a DVR;
\end{enumerate}
there exists a collection of isomorphisms $\{ \phi_T \}$ which:
\begin{enumerate}
\item satisfies the cocycle condition; 
\item extends the collection of $\ref{near existence DD}$ (so is compatible with base change preserving generic disjointness); and
\item is compatible with arbitrary specialization.
\end{enumerate}
Therefore, in the notation of $\ref{main thm intro}$, the incidence bundle $\cal{L} \in \pic(Y_0)$ lifts to an element $(\cal{L}, \phi) \in \pic(Y_\bullet)$. 
\end{theorem}

\begin{proof}  This follows from the general lemma $\ref{dense specializations}$, the preceding result $\ref{near existence DD}$, and the hypothesis that $T$ is seminormal.
\end{proof}

Now we conclude the proof of $\ref{main thm intro}$.

\begin{theorem}[descent property]  \label{end game} The element $(\cal{L}, \phi = \{ \phi_T \} )$ of $\ref{existence of DD}$ has the descent property: for any cycle $(z,W) \in \overline{U} \subset \scr{C}_a \times \scr{C}_b$ there exists an open subscheme $V \subset \overline{U}$ containing $(z,W)$ and an isomorphism $t : \cal{L} \arrowvert_{\pi_0^{-1}(V)} \cong \cal{O}_{\pi_0^{-1}(V)}$ which is compatible with $\phi$.  Therefore, the incidence bundle $\cal{L}$ descends to $\overline{U}$.
 \end{theorem}
 
\begin{proof} This structure is built into the definition of the descent datum $\phi = \{ \phi_T \}$.  By $\ref{diagonal}$ we may assume $W$ is fixed.  Suppose $z$ and $W$ are disjoint.  Then $W$ is disjoint from all cycles in a neighborhood $V$ of $z$, and also from all subschemes in $V_0: = \pi_0^{-1}(V) \subset \scr{H}_a$.  We let $Z_{V_0} \into P \times V_0$ denote the corresponding family.  On $V_0$ we use the canonical trivialization $\varphi^{Z_{V_0}}_{V_0} :  f_{V_0} (\cal{O}_{Z_{V_0}} , \cal{O}_W) :=  \det_{V_0} (\textbf{R} \pi_\ast (\cal{O}_{Z_{V_0}} \dotimes \cal{O}_W)) \cong \cal{O}_{V_0}$ induced by the acyclicity of $\cal{O}_{Z_{V_0}} \dotimes \cal{O}_W$.  Then by our definition of $\phi$ on disjoint subschemes, the following diagram commutes:

$$\xym{f_{V_0} (\cal{O}_{Z_{V_0}} , \cal{O}_W) \ar[d]^-{\phi^{Z,Z'}_{V_0}} \ar[r]^-{\varphi^{Z_{V_0}}_{V_0}} & \cal{O}_{V_0} \ar[d]^= \\
f_{V_0} (\cal{O}_{Z'_{V_0}} , \cal{O}_W) \ar[r]^-{\varphi^{Z'_{V_0}}_{V_0}} & \cal{O}_{V_0} \\}$$

\medskip \noi For a pair $(z, W)$ in the incidence locus,  choose a collection of short exact sequences as in $\ref{moving lemma}$ moving $W$ to a rationally equivalent $\alpha$ such that $z \cap \alpha= \emptyset$.  Then also $z' \cap \alpha = \emptyset$ for $z'$ in a neighborhood $V \ni z$, and $\alpha$ is disjoint from all subschemes parameterized by $V_0 := \pi_0^{-1}(V)$.  Then we define $t$ to be the trivialization induced by the move, then the acyclicity of $\cal{O}_{Z_{V_0}} \dotimes \cal{O}_{W_i}$ (for all $W_i \in \supp(\alpha)$):

$$t : f_{V_0} (\cal{O}_{Z_{V_0}} , \cal{O}_W) \cong f_{V_0} (\cal{O}_{Z_{V_0}} , \alpha) \xrightarrow{\varphi_{V_0}^{Z_{V_0}}} \cal{O}_{V_0}.$$
This is compatible with $\phi$ by $\ref{proper compat moving}$.  Verifying the descent property guarantees effectiveness by $\ref{Chow exists}$ and $\ref{sn pic inj}$.
\end{proof}

\subsection{Conclusion of the proof of $\ref{nice main thm}$.}  Finally we verify the properties stated in $\ref{nice main thm}$.

\medskip \textbf{Descent of rational section.} To see we have actually constructed a Cartier divisor in $\overline{U} \subset \scr{C}_a \times \scr{C}_b$ supported on the incidence locus, consider the diagram whose vertical arrows are the restriction maps:
$$\xym{
\pic(\overline{U}) \ar[d] \ar[r] & \pic(Y_\bullet) \ar[d] \\
\pic(U) \ar[r] & \pic({(\pi^{-1}(U))}_\bullet) \\ }$$
The arrow in the bottom row is injective by $\ref{sn pic inj}$.  On the locus of disjoint subschemes, the isomorphism constructed in $\ref{varphi disjoint}$, $\varphi : \cal{L} \arrowvert_{U_0} \cong \cal{O}_{U_0}$, is an isomorphism of pairs $\cal{L} \arrowvert_{{(\pi^{-1}(U))}_\bullet} =(\cal{L} \arrowvert_{U_0}, \phi \arrowvert_{U_0 \times_U U_0}) \cong (\cal{O}_{U_0}, id)$.  Hence by the injectivity of the bottom row, the trivialization $\varphi$ descends to $U \subset \overline{U}$.

\medskip \textbf{Restriction to $U'$ is effective.}  To check that the restriction $D |_{U'}$ is effective, we may replace $U'$ with its normalization ${U'}^\nu$.  Then we may replace ${U'}^\nu$ with the local ring of some depth 1 point $t$ on ${U'}^\nu$.  By the assumption that the universal cycles intersect properly, over a given component $C \subset U'$, the incidence has dimension $\dim(C) -1$.  This is preserved by the finite base change ${U'}^\nu \to U'$.

Suppose first the incidence is generically finite onto its image.  Then $\ref{varphi proper HC}$ applies, and the coefficient of $t$ in $D |_{{U'}^\nu}$ is a sum of intersection multiplicities of properly intersecting components (weighted with positive coefficients).  If the incidence dominates $t$, this coefficient is positive by \cite[V.C.Thm.1(b)]{Ser}; in any case the coefficient is nonnegative.

If the incidence has generic positive dimension over its image, then its image must have dimension $\leq \dim(C)-2$.  Hence in this case the associated coefficient is 0.

\medskip \textbf{Intersection multiplicity.} On the Chow varieties we have the incidence bundle $\cal{M}$ and its rational section over the locus of disjoint cycles, giving the Cartier divisor $D \into \overline{U}$.  This pulls back via the Hilbert-Chow morphism $\pi :Y_0 \to \overline{U}$ to the determinant line bundle $\cal{L}$ and its rational section over the locus of disjoint subschemes.  Our goal is to relate the order of vanishing of a local defining equation of $D$, to intersection numbers.  So let $s_D$ be the canonical (rational) section of the line bundle $\cal{O}_{\overline{U}}(D)$.

If $g: T \to \overline{U}$ is a morphism from the spectrum of a discrete valuation ring $R \supset k$ (corresponding to cycles $Z,W$), there exists a discrete valuation ring $R'$ which is finite over $R$, and such that the composition $g' : T' := \spec R' \to T \to \overline{U}$ factors through $Y_0$.  (Note that if we start with a specialization from a generic point of $\overline{U}$, we can find a component of the Hilbert scheme so that no generic extension is necessary.)  If $\ord g'^* (s_D) = \deg (Z_{T'} \cdot W_{T'})$, it follows that $\ord g^* (s_D) = \deg (Z \cdot W)$.  Thus we may assume our specialization factors through the Hilbert scheme, corresponding to subschemes $\widetilde{Z}, \widetilde{W}$ such that $[\widetilde{Z}] =Z$ and $[\widetilde{W}] = W$.  Now we assume disjointness over the generic point $\eta \in T$, i.e., $g(\eta) \in U$.  Let $t \in T$ denote the closed point.

First we have:
$$\ord_{T}(s_D) = \ord_{T}(s_{\pi^*D}) = \sum_{p} {(-1)}^p \ell_{t} ( \cal{H}^p(\textbf{R} \pi_\ast (\cal{O}_{\widetilde{Z}} \dotimes \cal{O}_{\widetilde{W}})))$$
since each $\cal{H}^p(\textbf{R} \pi_\ast (\cal{O}_{\widetilde{Z}} \dotimes \cal{O}_{\widetilde{W}}))$ is a torsion $T$-module, and by \cite[Thm.3(v)]{KM}.

Since the scheme $P_T$ is smooth, the filtration of the $K_0$-groups by dimension is compatible with multiplication, thus $\cal{O}_{\widetilde{Z}} \dotimes \cal{O}_{\widetilde{W}}$ and $\textbf{R} \pi_\ast (\cal{O}_{\widetilde{Z}} \dotimes \cal{O}_{\widetilde{W}})$ are classes of dimension zero.  Then $\sum_{p} {(-1)}^p \ell_{t} ( \cal{H}^p(\textbf{R} \pi_\ast (\cal{O}_{\widetilde{Z}} \dotimes \cal{O}_{\widetilde{W}})))$ is equal to the degree of the $K_0$-classes $\cal{O}_{\widetilde{Z}} \dotimes \cal{O}_{\widetilde{W}}$ and $\textbf{R} \pi_\ast (\cal{O}_{\widetilde{Z}} \dotimes \cal{O}_{\widetilde{W}})$.  Note also the refined class $Z \cdot W$ is of the expected dimension, i.e., $Z \cdot W \in A_0(P_T)$.

We have $\cal{O}_{\widetilde{Z}} \dotimes \cal{O}_{\widetilde{W}} = \sum_i {(-1)}^i [\Tor_i^{P_{T}} (\cal{O}_{\widetilde{Z}}, \cal{O}_{\widetilde{W}})] \in K_0(P_T)$.  The degree of this class is computed by \cite[20.4]{Ful}: it is simply the degree of the refined class $Z \cdot W \in A_0(P_T)$, since the terms of dimension $<0$ necessarily vanish.

\bibliography{curvesreferences}{}
\bibliographystyle{plain}

\end{document}